\documentclass[30pt]{article}
\usepackage{amsfonts}
\usepackage{mathrsfs}
\usepackage{amsmath}
\usepackage{amssymb}
\usepackage{graphicx}
\usepackage{geometry}
\usepackage{epic}
\usepackage{epstopdf}
\usepackage{ytableau}
\usepackage{color}
\usepackage{cite}
\usepackage{mathtools}
\usepackage{amsthm}
\renewcommand{\paragraph}{\roman{paragraph}}

\usepackage{bm}
\usepackage{bbm}
\usepackage{hyperref}
\usepackage{tikz}
\usepackage{enumerate}
\usetikzlibrary{arrows,shapes,positioning}
\usetikzlibrary{decorations.markings}
\tikzstyle arrowstyle=[scale=1]
\tikzstyle directed=[postaction={decorate,decoration={markings, mark=at position .65 with {\arrow[arrowstyle]{stealth}}}}]
\tikzstyle reverse directed=[postaction={decorate,decoration={markings, mark=at position .65 with {\arrowreversed[arrowstyle]{stealth};}}}]

\newtheorem{claim}{Claim}[section]
\newtheorem{theorem}{Theorem}[section]
\newtheorem{corollary}[theorem]{Corollary}

\newtheorem{question}[theorem]{Question}

\newtheorem{lemma}[theorem]{Lemma}

\newtheorem{proposition}[theorem]{Proposition}
\newtheorem{Remark}[theorem]{Remark}

\newcommand{\Gaussbinom}{\genfrac{[}{]}{0pt}{}}

\usepackage{enumerate}
\usepackage{amssymb}
\usepackage{graphics}
\usepackage{graphicx}
\usepackage{subfigure}
\usepackage{epsfig}
\usepackage{amsmath}
\usepackage{bm}
\usepackage{lineno}
\usepackage{url}
\usepackage{float}
\usepackage{color}

\begin{document}

\title{Inverse problems of the Erd\H{o}s-Ko-Rado type theorems for families of vector spaces and permutations}

\author{Xiangliang Kong$^{\text{a}}$, Yuanxiao Xi$^{\text{b}}$, Bingchen Qian$^{\text{b}}$ and Gennian Ge$^{\text{a,}}$\thanks{ Corresponding author. Email address:  gnge@zju.edu.cn.
  The research of G. Ge was supported by the National Natural Science Foundation of China under Grant No. 11971325, National Key Research and Development Program of China under Grant Nos. 2020YFA0712100 and 2018YFA0704703, and Beijing Scholars Program.}\\
  \footnotesize $^{\text{a}}$ School of Mathematical Sciences, Capital Normal University, Beijing, 100048, China\\
  \footnotesize $^{\text{b}}$ School of Mathematical Sciences, Zhejiang University, Hangzhou 310027, Zhejiang, China}
\date{}
\maketitle

\begin{abstract}
  Ever since the famous Erd\H{o}s-Ko-Rado theorem initiated the study of intersecting families of subsets, extremal problems regarding intersecting properties of families of various combinatorial objects have been extensively investigated. Among them, studies about families of subsets, vector spaces and permutations are of particular concerns.

  Recently, the authors proposed a new quantitative intersection problem for families of subsets: For $\mathcal{F}\subseteq {[n]\choose k}$, define its \emph{total intersection number} as $\mathcal{I}(\mathcal{F})=\sum_{F_1,F_2\in \mathcal{F}}|F_1\cap F_2|$. Then, what is the structure of $\mathcal{F}$ when it has the maximal total intersection number among all families in ${[n]\choose k}$ with the same family size? In \cite{KG2020}, the authors studied this problem and characterized extremal structures of families maximizing the total intersection number of given sizes.

  In this paper, we consider the analogues of this problem for families of vector spaces and permutations. For certain ranges of family size, we provide structural characterizations for both families of subspaces and families of permutations having maximal total intersection numbers. To some extent, these results determine the unique structure of the optimal family for some certain values of $|\mathcal{F}|$ and characterize the relation between having maximal total intersection number and being intersecting. Besides, we also show several upper bounds on the total intersection numbers for both families of subspaces and families of permutations of given sizes.


  \smallskip


\end{abstract}

\section{Introduction}

For a positive integer $n$, let $[n]=\{1,2,\ldots,n\}$ and ${[n]\choose k}$ denote the collection of all $k$-element subsets of $[n]$. A family $\mathcal{F}\subseteq {[n]\choose k}$ is called \emph{intersecting} if any two of its members share at least one common element. The classic Erd\H{o}s-Ko-Rado theorem states that if $n\geq 2k+1$, an intersecting family has size at most ${n-1\choose k-1}$; if the equality holds, the family must be consisted of all $k$-subsets of $[n]$ containing a fixed element.
Inspired by this cornerstone result in extremal set theory, there have been a great deal of extensions and variations. As two major extensions, intersection problems for families of permutations and families of subspaces over a given finite field have drawn lots of attentions in these years.

Let $\mathbb{F}_q$ be the finite field with $q$ elements and $V$ be an $n$-dimensional vector space over $\mathbb{F}_q$. Denote $\Gaussbinom{V}{k}$ as the collection of all $k$-dimensional subspaces of $V$ and for $t\geq 1$, $\mathcal{F}\subseteq \Gaussbinom{V}{k}$ is called \emph{$t$-intersecting} if $\dim(F\cap F')\geq t$ holds for all $F,F'\in \mathcal{F}$. In 1986, using spectra method, Frankl and Wilson \cite{FW1986} proved the following analogue result of Erd\H{o}s-Ko-Rado theorem for $t$-intersecting family of subspaces of $V$. Since then, many other kinds of intersection problems for families of subspaces have been studied, for examples, see \cite{CP2010,Tanaka2006,BBCFMBS2010}.
\begin{theorem}(\cite{FW1986})\label{EKR_subspace}
Let $n\geq 2k$ and $k\geq t>0$ be integers and let $\mathcal{F}\subseteq \Gaussbinom{V}{k}$ be a $t$-intersecting family, then $|\mathcal{F}|\leq \Gaussbinom{n-t}{k-t}_q$. Moreover, when $n\geq 2k+1$, the equality holds if and only if $\mathcal{F}$ is the family of $k$-dim subspaces containing a fixed $t$-dim subspace.
\end{theorem}

Let $S_n$ be the symmetric group of all permutations of $[n]$ and for $t\geq 1$, a subset $\mathcal{F}\subseteq S_n$ is called \emph{$t$-intersecting} if there exist $t$ distinct integers $i_1,i_2,\ldots,i_t\in[n]$ such that $\sigma(i_j)=\tau(i_j)$ for $j=1,2,\ldots,t$ and $\sigma,\tau\in \mathcal{F}$. Let $\mathcal{C}_{{i_1\rightarrow j_1},\ldots,{i_t\rightarrow j_t}}=\{\sigma\in S_n: \sigma(i_s)=j_s, \text{~for~} s=1,\ldots,t\}$, if $i_1,\ldots,i_t$ are distinct and $j_1,\ldots,j_t$ are distinct, then $\mathcal{C}_{{i_1\rightarrow j_1},\ldots,{i_t\rightarrow j_t}}$ is a coset of the stabilizer of $t$ points, which is referred as a \emph{$t$-coset}. In \cite{DF1977}, Deza and Frankl proved the following theorem for $1$-intersecting family of permutations.
\begin{theorem}(\cite{DF1977})\label{EKR_permutation}
For any positive integer $n$, if $\mathcal{F}\subseteq S_n$ is $1$-intersecting, then $|\mathcal{F}|\leq (n-1)!$.
\end{theorem}
Clearly, a $1$-coset is a $1$-intersecting family of size $(n-1)!$. Deza and Frankl \cite{DF1977} conjectured that the $1$-cosets are the only $1$-intersecting families of permutations with size $(n-1)!$. This conjecture was first confirmed by Cameron and Ku \cite{CK2003} and independently by Larose and Malvenuto \cite{LM2004}. As for $t$-intersecting families of permutations when $t\geq 2$, in the same paper, Deza and Frankl also conjectured that the $t$-cosets are the only largest $t$-intersecting families in $S_n$ provided $n$ is large enough. Using eigenvalue techniques together with the representation theory of $S_n$, Ellis, Friedgut and Pilpel \cite{EFP2011} proved this conjecture.


Following the path led by Erd\H{o}s, Ko and Rado, the above studies of intersections problems about subspaces and permutations share a same type of flavour: Given a family of subspaces or permutations with some certain kind of intersecting property, people try to figure out how large this family can be. Once the maximal size of the family with given intersecting property is determined, people turn to an immediate inverse problem --- characterizing the structure of the extremal family. This gives rise to further studies of the stability and supersaturation for extremal families. In recent years, there have been a lot of works concerning this kind of inverse problems, for examples, see \cite{BDLST2019,DGS2016,BDDLS2015,KKK2012,Ellis2011,FKR2016,Russell2012,GS2020,DT2016}.

In this paper, with the same spirit, we consider an inverse problem for intersecting families of subspaces and permutations from another point of view. Instead of being intersecting, we assume that the family possesses a certain property that ``maximizes'' the intersections quantitatively. The study of this kind of inverse problem was first initiated by the first and the last authors in \cite{KG2020}, where families of subsets were investigated.

To state the problem formally, first, we introduce the notion \emph{total intersection number} of a family. Let $X$ be the underlying set with finite members, $X$ can be ${[n]\choose k}$, or $\Gaussbinom{V}{k}$ for an $n$-dimensional space $V$ over $\mathbb{F}_q$, or $S_n$. Consider a family $\mathcal{F}\subseteq X$, the \emph{total intersection number} of $\mathcal{F}$ is defined by
\begin{flalign}\label{basic_id}
\mathcal{I}(\mathcal{F})=\sum\limits_{A\in \mathcal{F}}\sum\limits_{B\in \mathcal{F}}int(A,B),
\end{flalign}
where $int(A,B)$ has different meanings for different $X$s. When $X={[n]\choose k}$, $int(A,B)=|A\cap B|$; when $X=\Gaussbinom{V}{k}$, $int(A,B)=\dim(A\cap B)$; when $X=S_n$, $int(A,B)=|\{i\in [n]:A(i)=B(i)\}|$. Moreover, we denote
\begin{equation}\label{basic_id2}
\mathcal{MI}(X,\mathcal{F})=\max_{\mathcal{G}\subseteq X, |\mathcal{G}|=|\mathcal{F}|}\mathcal{I}(\mathcal{G})
\end{equation}
as the maximal total intersection number among all families in $X$ with the same size of $\mathcal{F}$ and we denote it as $\mathcal{MI}(\mathcal{F})$ for short if $X$ is clear. Similarly, for two families $\mathcal{F}_1$ and $\mathcal{F}_2$ in $X$, the total intersection number between $\mathcal{F}_1$ and $\mathcal{F}_2$ is defined as
\begin{equation}\label{basic_id3}
\mathcal{I}(\mathcal{F}_1,\mathcal{F}_2)=\sum_{A\in\mathcal{F}_1}\sum_{B\in\mathcal{F}_2}int(A,B).
\end{equation}
Clearly, we have $\mathcal{I}(\mathcal{F},\mathcal{F})=\mathcal{I}(\mathcal{F})$.

Certainly, the value of $\mathcal{I}(\mathcal{F})$ reveals the level of intersections among the members of $\mathcal{F}$: the larger $\mathcal{I}(\mathcal{F})$ is, the more intersections there will be in $\mathcal{F}$. For an integer $t\geq 1$, note that being $t$-intersecting also indicates that $\mathcal{F}$ possesses a large amount of intersections, therefore, it is natural to ask the relationship between being $t$-intersecting and having large $\mathcal{I}(\mathcal{F})$:
\begin{question}\label{question1}
For $t\geq 1$ and $n$ large enough, denote $M(X,t)$ as the maximal size of the $t$-intersecting family in $X$. Let $\mathcal{F}\subseteq X$ with size $M(X,t)$, if $\mathcal{I}(\mathcal{F})=\mathcal{MI}(\mathcal{F})$, is $\mathcal{F}$ a $t$-intersecting family? Or, if $\mathcal{F}$ is a maximal $t$-intersecting family in $X$, do we have $\mathcal{I}(\mathcal{F})=\mathcal{MI}(\mathcal{F})$?
\end{question}


In \cite{KG2020}, by taking $X={[n]\choose k}$, we show that when $|\mathcal{F}|={n-t\choose k-t}$ and $\mathcal{I}(\mathcal{F})=\mathcal{MI}(\mathcal{F})$, the full $t$-star (the family consisting of all $k$-sets in ${[n]\choose k}$ containing $t$ fixed elements) is indeed the only structure of $\mathcal{F}$, which answers the Question \ref{question1} for the case $X={[n]\choose k}$. In this paper, when $X=\Gaussbinom{V}{k}$ and $\dim(V)=n$ is large enough, we obtain similar results for general $t\geq1$; when $X=S_n$, we answer the Question \ref{question1} for the case $t=1$. Noticed that the property of having maximal total intersection number can be considered for families of any size. Actually, we can ask the following more general question:
\begin{question}\label{question}
For a family $\mathcal{F}\subseteq X$, if $\mathcal{I}(\mathcal{F})=\mathcal{MI}(\mathcal{F})$, what can we say about its structure?
\end{question}

Aiming to solve these questions, we provide structural characterizations for optimal families satisfying $\mathcal{I}(\mathcal{F})=\mathcal{MI}(\mathcal{F})$ for both of the cases when $X=\Gaussbinom{V}{k}$ and $X=S_n$. Moreover, we also obtain some upper bounds on $\mathcal{MI}(\mathcal{F})$ for several ranges of $|\mathcal{F}|$ for both cases. The detailed descriptions of our results will be shown in the following subsection.

\subsection{Our results}

When $X=\Gaussbinom{V}{k}$, through combinatorial arguments, we have the following theorem which shows the main structure of the optimal family $\mathcal{F}\subseteq X$ with $|\mathcal{F}|$ not much larger than $\Gaussbinom{n-t}{k-t}_q$.

\begin{theorem}\label{stabilityforsub}
Given positive integers $1\leq t<k$ and $n\geq (4k+4)^2\Gaussbinom{k}{t}_q^2$, let $\mathcal{F}$ be a family of $k$-dim subspaces of $V$ with size $|\mathcal{F}|=\delta\Gaussbinom{n-t}{k-t}_q$ for some $\delta\in[\frac{(4k+4)^2n}{q^{n-k}},1+\frac{1}{96t\ln{q}(k+1)}]$ satisfying $\mathcal{I}(\mathcal{F})=\mathcal{MI}(\mathcal{F})$. Then, when $\delta\leq 1$, $\mathcal{F}$ is contained in a full $t$-star and when $\delta >1$, $\mathcal{F}$ contains a full $t$-star.
\end{theorem}

When $X=S_n$, for an integer $s>\frac{1}{2}(n-1)!$, consider the subfamilies of $X$ consisting of $\lfloor\frac{s}{(n-1)!}\rfloor$ pairwise disjoint $1$-cosets and $s-\lfloor\frac{s}{(n-1)!}\rfloor(n-1)!$ permutations from another $1$-coset disjoint with all the former $1$-cosets. We denote $\mathcal{T}(n,s)$ as the family of this form with size $s$ with maximal total intersection number. Using eigenvalue techniques together with the representation theory of $S_n$, we prove that families of permutations of size $\Theta((n-1)!)$ having large total intersection numbers are close to the union of $1$-cosets.

\begin{theorem}\label{theremovallemma}
For a sufficiently large integer $n$, there exist positive constants $C_0$ and $c$ such that the following holds. For integer $0\leq k\leq \frac{n-1}{2}$, let $\varepsilon\in(-\frac{1}{2},\frac{1}{2}]$ and $\delta\geq 0$ such that $\max\{|\varepsilon|,\delta\}\leq ck$. If $\mathcal{F}$ is a subfamily of $S_n$ with size $(k+\varepsilon)(n-1)!$ and $\mathcal{I}(\mathcal{F})\geq \mathcal{I}(\mathcal{T}(n,|\mathcal{F}|))-\delta((n-1)!)^2$, then there exists some $\mathcal{G}\subseteq S_n$ consisting of $k$ $1$-cosets such that
\begin{equation*}
|\mathcal{F}\Delta\mathcal{G}|\leq C_0\left(\sqrt{2k(|\varepsilon|+\delta)}+\frac{k}{n}\right)|\mathcal{F}|.
\end{equation*}
Moreover, when $\varepsilon=\delta=0$, $\mathcal{F}=\mathcal{G}_0$ for some $\mathcal{G}_0\subseteq S_n$ consisting of $k$ pairwise disjoint $1$-cosets.
\end{theorem}

Moreover, using linear programming method over association schemes, we also have the following upper bounds on $\mathcal{MI}(\mathcal{F})$.

\begin{theorem}\label{subspace_intersection}
Given positive integers $n$, $k$, $M$ with $k\leq n$ and $M\leq \Gaussbinom{n}{k}_q$, for $\mathcal{F}\subseteq \Gaussbinom{V}{k}$ with $|\mathcal{F}|=M$, we have
\begin{flalign}\label{lower_bound_0}
\mathcal{MI}(\mathcal{F})\le\left(\frac{\Gaussbinom{n}{k}_q}{M}-\Gaussbinom{n}{1}_q\right)\frac{qM^2\Gaussbinom{k}{1}_q\Gaussbinom{n-k}{1}_q}{\Gaussbinom{n}{1}_q\left(\Gaussbinom{n}{1}_q-1\right)}+kM^2,
\end{flalign}
especially, when $n\ge 2k$ and $M\leq\Gaussbinom{n-1}{k-1}_q$, we have
\begin{flalign}\label{lower_bound_space_2}
\mathcal{MI}(\mathcal{F})\leq\left[\frac{\Gaussbinom{n}{k}_q}{M}-\frac{(q^n-1)(q^{n-1}-1)}{(q-1)(q^k-1)}\right]\frac{M^2(q^k-1)(q^{k-1}-1)(q^{n-k}-1)}{(q^n-1)(q^{n-1}-1)(q^{n-2}-1)}+kM^2.
\end{flalign}
\end{theorem}
%
%
%

\begin{theorem}\label{lpbound}
Given positive integers $n$ and $M\leq n!$, for $\mathcal{F}\subseteq S_n$ with $|\mathcal{F}|=M$, we have
\begin{flalign*}
\mathcal{MI}(\mathcal{F})\leq\frac{M^2}{n-1}\left(\frac{n!}{M}+n-2\right).
\end{flalign*}
\end{theorem}


\subsection{Notations and Outline}

Throughout this paper, we shall use the following standard mathematical notations .

\begin{itemize}
  \item Denote $\mathbb{N}$ as the set of all nonnegative integers. For any $n\in\mathbb{N}\setminus\{0\}$, let $[n]=\{1,2,\ldots,n\}$. For any $a,b\in\mathbb{N}$ such that $a\leq b$, let $[a,b]=\{a,a+1,\ldots,b\}$.
  \item Given finite set $S$ and any positive integer $k$, denote ${S\choose k}$ as the family of all $k$-subsets of $S$ and $2^S$ as the family of all subsets of $S$.
  \item For a given prime power $q$ and a positive integer $n$, we denote $\mathbb{F}_q$ as a finite field with $q$ elements and $\mathbb{F}_q^n$ as the $n$-dimensional vector space over $\mathbb{F}_q$. Moreover, for a vector $\mathbf{x}$ with length $n$, we denote $\mathbf{x}_i$ as the $i_{th}$ position of $\mathbf{x}$ for $1\leq i\leq n$.
  \item For two subspaces $V_1,V_2\subseteq \mathbb{F}^{n}_q$, we denote $V_1+V_2$ as the sum of these two subspaces and $V_1/V_2$ as the quotient subspace of $V_1$ by $V_2$. If $V_1\cap V_2=\{\mathbf{0}\}$, we denote $V_1\oplus V_2$ as the direct sum of $V_1,V_2$. Moreover, we have $\dim(V_1+ V_2)=\dim(V_1)+\dim(V_2)-\dim(V_1\cap V_2)$ and $\dim(V_1/V_2)=\dim(V_1)-\dim(V_1\cap V_2)$.
  \item For a given prime power $q$ and positive integers $n$, $k$ with $k\leq n$, the \emph{Gaussian binomial coefficient} $\Gaussbinom{n}{k}_q$ is defined by
      \begin{flalign*}
      \Gaussbinom{n}{k}_q=\prod\limits_{i=0}^{k-1}\frac{q^{n-i}-1}{q^{k-i}-1}.
      \end{flalign*}
      Usually, the $q$ is omitted when it is clear.
  \item For a given family $\mathcal{F}$ in $\Gaussbinom{V}{k}$ and a $t$-dim subspace $U\subseteq V$, we denote $\mathcal{F}(U)=\{F\in \mathcal{F}: U\subseteq F\}$ as the subfamily in $\mathcal{F}$ containing $U$ and $\deg_{\mathcal{F}}(U)=|\mathcal{F}(U)|$ is called the degree of $U$ in $\mathcal{F}$. 

\end{itemize}

The remainder of the paper is organized as follows. In Section 2, we will introduce some basic notions and known results on general association schemes, representation theory of $S_n$ and spectra of Cayley graphs on $S_n$. Moreover, we also include some preliminary lemmas for the proof of our main results. In Section 3, we consider families of vector spaces and prove Theorem \ref{stabilityforsub}. In Section 4, we consider families of permutations and prove Theorem \ref{theremovallemma}. And we prove Theorem \ref{subspace_intersection} and Theorem \ref{lpbound} in Section 5. Finally, we will conclude the paper and discuss some remaining problems in Section 6.

\section{Preliminaries}

In this section, we will introduce some necessary notions and related results to support proofs of our theorems. First, we will introduce some notions about general association schemes, which are crucial for the proof of the upper bounds on $\mathcal{MI}(X,\mathcal{F})$ for $X=\Gaussbinom{V}{k}$ and $X=S_n$. Then, we shall give a brief introduction on the representation theory of $S_n$.  Finally, we will review some known results about spectra of Cayley graphs on $S_n$. Readers familiar with these parts are invited to skip corresponding subsections. Based on these results, we will provide some new estimations about eigenvalues of Cayley graphs on $S_n$ for the proof of Theorem \ref{theremovallemma}.

\subsection{Association schemes}\label{association_schemes}
Association scheme is one of the most important topics in algebraic combinatorics, coding theory, etc. Many questions concerning distance-regular graphs are best solved in this framework, see \cite{B1971},\cite{B1974}. In 1973, by performing linear programming methods on specific association schemes, Delsarte \cite{DP1973} proved many of the sharpest bounds on the size of a code, which demonstrated the power of association schemes in coding theory. Since then, association schemes have been widely studied and related notions have also been extended to other objects, such as equiangular lines and special codes, etc. In this subsection, we only include some basic notions about association schemes. For more details about association schemes, we recommend \cite{DP1973} and \cite{GC2016} as standard references.


Let $X$ be a finite set with $v$ $(v\ge 2)$ elements, and for any integer $s\ge 1$, let $\mathcal{R}=\{R_0,R_1,\ldots,R_s\}$ be a family of $s+1$ relations on $X$. The pair $(X,\mathcal{R})$ is called an association scheme with $s$ classes if the following three conditions are satisfied:
\begin{itemize}
  \item [1.] The set $\mathcal{R}$ is a partition of $X\times X$ and $R_0$ is the diagonal relation, i.e., $R_0=\{(x,x)|~x\in X\}$.
  \item [2.] For $i=0,1,\ldots,s$, the inverse $R_i^{-1}=\{(y,x)|~(x,y)\in R_i\}$ of the relation $R_i$ also belongs to $\mathcal{R}$.
  \item [3.] For any triple of integers $i,j,k\in\{ 0,1,\ldots,s\}$, there exists a number $p_{i,j}^{(k)}=p_{j,i}^{(k)}$ such that, for all $(x,y)\in R_k$:
       $$|\{z\in X|~(x,z)\in R_i,~(z,y)\in R_j\}|=p_{i,j}^{(k)}.$$
       And $p_{i,j}^{(k)}$s are called the $intersection$ $numbers$ of $(X,\mathcal{R})$.
\end{itemize}
For relation $R_i\in \mathcal{R}$, the adjacency matrix of $R_i$ is defined as follows:
\begin{flalign*}
A_{i}(x,y)=\left\{\begin{array}{ll}1,&~(x,y)\in R_i,\\0,&~(x,y)\notin R_i. \end{array}\right.
\end{flalign*}
The space consisting of all complex linear combinations of the matrices $\{A_0,\ldots,A_s\}$ in an association scheme $(X,\mathcal{R})$ is called a $Bose$-$Mesner$ $algebra$. Moreover, denote $J$ as the $v\times v$ matrix with all entries 1, there is a set of pairwise orthogonal idempotent matrices $\{B_0=\frac{J }{v},\ldots,B_s\}$, which forms another basis of this Bose-Mesner algebra. The relations between $\{A_r\}_{r=0}^s$ and $\{B_r\}_{r=0}^s$ are shown as follows:
\begin{flalign}\label{relationship}
A_i=\sum\limits_{j=0}^sP_{i}(j)B_j,~i=0,\ldots,s;~~~~~B_j=\frac{1}{v}\sum\limits_{i=0}^sQ_j(i)A_i, ~j=0,\ldots,s,
\end{flalign}
where $P_i(0),\ldots,P_i(s)$ are the eigenvalues of $A_i$, which are called the $eigenvalues$ of the association scheme; and $Q_j(i)$ are known as $dual$ $eigenvalues$ of the association scheme. Usually, $v_i:=P_i(0)$ denotes the number of 1's in each row of $A_i$ and $u_j:=Q_j(0)=tr(B_j)$. According to \cite{GC2016}, for $1\leq i,j\leq s$, $P_i(j)$s and $Q_{j}(i)$s have the following relation:
\begin{flalign}\label{relation_PQ}
\frac{\overline{P_i(j)}}{v_i}=\frac{Q_j(i)}{u_j}.
\end{flalign}

%



Let $\mathcal{R}=\{R_0,R_1,\ldots,R_s\}$ be a set of $s+1$ relations on $X$ of an association scheme. Given a subset $Y\subseteq X$ with $|Y|=M$, the $inner~distribution$ of $Y$ with respect to $\mathcal{R}$ is an $(s+1)$-tuple $\mathbf{a}=(a_0,\ldots,a_s)$ of nonnegative rational numbers $a_i$ $(0\leq i\leq s)$ given by
\begin{flalign}\label{inner_distribution}
a_i=\frac{|R_i\cap (Y\times Y)|}{M}.
\end{flalign}
Clearly, we have $a_0=1$ and $\sum_{i=0}^s a_i=|Y|$.

Moveover, let $\mathbf{u}$ be the indicator vector of $Y$ with respect to $X$, i.e., $\mathbf{u}_x=1$, if $x\in Y$ and $\mathbf{u}_x=0$, if $x\notin Y$. Then, $(\ref{inner_distribution})$ can be rewritten as
\begin{flalign}\label{another_form}
a_i=\frac{1}{M}\mathbf{u}A_i\mathbf{u}^{T}.
\end{flalign}
Besides, for $0\leq j\leq s$, define
\begin{flalign}\label{A_represent_B}
b_j=\frac{v}{M^2}\mathbf{u}B_j\mathbf{u}^{T},
\end{flalign}
and $\mathbf{b}=(b_0,\ldots,b_s)$ as the $dual$ $distribution$ of $Y$. By combining $(\ref{relationship})$ and $(\ref{A_represent_B})$ together, we have the following lemma which provides a linear relationship between $a_i$s and $b_{j}$s.
\begin{lemma}\label{B_represent_A}
Given an association scheme $(X,\mathcal{R})$ with $s$ classes and $|X|=v$. Let $Y\subseteq X$ with size $M$, then for $\{a_0,\ldots,a_s\}$ and $\{b_0,\ldots,b_s\}$ defined in $(\ref{inner_distribution})$ and $(\ref{A_represent_B})$ respectively,  we have
\begin{flalign*}
a_i=\frac{M}{v}\sum\limits_{j=0}^sb_jP_i(j), ~~i=0,1,\ldots,s.
\end{flalign*}
\end{lemma}

As a consequence of Lemma \ref{B_represent_A}, we have the following properties of $\{b_j:0\leq j\leq s\}$.
\begin{lemma}(\cite{MS1977}, Theorem 12 in Section 6, Chapter 21)\label{B'_k}
Given an association scheme $(X,\mathcal{R})$ with $s$ classes and $|X|=v$. Let $Y\subseteq X$ with size $M$ and $\{b_0,\ldots,b_s\}$ be defined as $(\ref{A_represent_B})$, then $b_j\ge0$ for all $0\leq j\leq s$.
\end{lemma}
\begin{lemma}\label{property_B'}
With the same conditions as those in Lemma \ref{B'_k}, for $\{b_0,\ldots,b_s\}$, we have
\begin{flalign}
b_0=1 \text{~and~} \sum\limits_{j=0}^s b_j=\frac{v}{M}.\label{sumb}
\end{flalign}
\end{lemma}
\begin{proof}[Proof of Lemma \ref{property_B'}]
Since $B_0=J/v$, by the definition of $b_j$ in $(\ref{A_represent_B})$, we can obtain
\begin{flalign*}
b_0=\frac{1}{M^2}\mathbf{u}J\mathbf{u}^T=1.
\end{flalign*}
Note that $a_0=1$ and $P_0(j)=1$ for $0\leq j\leq s$, by taking $i=0$ in Lemma \ref{B_represent_A}, we can obtain
\begin{flalign*}
 \sum\limits_{j=0}^s b_j=\frac{v}{M}.
\end{flalign*}
\end{proof}

\subsection{Background on the representation theory of $S_n$}

A \emph{partition} of $n$ is a nonincreasing sequence of positive integers summing to $n$, i.e., a sequence $\lambda=(\lambda_1,\ldots,\lambda_l)$ with $\lambda_1\geq\cdots\geq\lambda_l$ and $\sum_{i=1}^{l}\lambda_i=n$, and we write $\lambda\vdash n$. The \emph{Young diagram} of $\lambda$ is an array of $n$ \emph{cells}, having $l$ left-justified rows, where row $i$ contains $\lambda_i$ cells. For example, the Young diagram of the partition $(3,2^2)$ is:
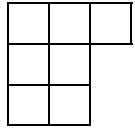
\begin{center}
    \begin{ytableau}
    \ \ & \ \ & \ \ \\
    \ \ & \ \ \\
    \ \ & \ \
    \end{ytableau}
\end{center}
If the array contains the numbers $\{1,\ldots,n\}$ in some order in place of dots, we call it \emph{$\lambda$-tableau}, for example,
\begin{center}
    \begin{ytableau}
    5 & 1 & 3 \\
    2 & 4 \\
    6 & 7 \\
    \end{ytableau}
\end{center}
is a $(3,2^2)$-tableau. Two $\lambda$-tableaux are said to be \emph{row-equivalent} if they have the same numbers in each row; if a $\lambda$-tableau $s$ has rows $R_1,\ldots,R_{l_1}\subseteq[n]$ and columns $C_1,\ldots,C_{l_2}\subseteq[n]$, then we let $R_s=S_{R_1}\times\cdots\times S_{R_{l_1}}$ be the row-stabilizer of $s$ and $C_s=S_{C_1}\times\cdots\times S_{C_{l_2}}$ be the column-stabilizer of $s$.

A \emph{$\lambda$-tabloid} is a $\lambda$-tableau with unordered row entries. Given a tableau $s$, denote $[s]$ as the tabloid it produces. For example, the $(3,2^2)$-tableau above produces the following $(3,2^2)$-tabloid:
\begin{align*}
&\{5~~1~~3 \} \\
&\{2~~4\}\\
&\{6~~7\}
\end{align*}
For given group $G$ and set $S$, denote $e$ as the identity in $G$. The left action of $G$ on $S$ is a function $G\times S\rightarrow S$ (denoted by $(g,x)\mapsto gx$) such that for all $x\in S$ and $g_1,g_2\in G:$
\begin{equation*}
ex=x~\text{and}~(g_1g_2)x=g_1(g_2x).
\end{equation*}
Now, consider the left action of $S_n$ on $X^{\lambda}$, the set of all $\lambda$-tabloids; let $M^{\lambda}=\mathbb{C}[X^{\lambda}]$ be the corresponding permutation module, i.e., the complex vector space with basis $X^{\lambda}$ and the action of $S_n$ on $\mathbb{C}[X^{\lambda}]$ linearly extended from the action of $S_n$ on $X^{\lambda}$. Given a $\lambda$-tableau $s$, the corresponding \emph{$\lambda$-polytabloid} is defined as
\begin{equation*}
e_s:=\sum_{\pi\in C_s}sgn(\pi)\pi[s].
\end{equation*}
We define the \emph{Specht module} $S^{\lambda}$ to be the submodule of $M^{\lambda}$ spanned by the $\lambda$-polytabloids:
\begin{equation*}
S^{\lambda}=\text{Span}\{e_s: s \text{~is a } \lambda\text{-tableau}\}.
\end{equation*}

As shown in \cite{EFP2011}, any irreducible representation $\rho$ of $S_n$ is isomorphic to some $S^{\lambda}$. This leads to a one to one correspondence between irreducible representations and partitions of $n$. In the following of this paper, for convenience, we shall write $[\lambda]$ for the equivalence class of the irreducible representation $S^{\lambda}$, $\chi_{\lambda}$ for the character $\chi_{S^{\lambda}}$ (The formal definition of the character of a representation will be presented in Section 2.3.1). 

Let $\lambda=(\lambda_1,\ldots,\lambda_{l_1})$ be a partition of $n$. If its Young diagram has columns of lengths $\lambda'_1\geq \lambda'_2\geq \cdots\geq \lambda'_{l_2}\geq 1$, then the partition $\lambda^{T}=(\lambda'_1,\dots,\lambda'_{l_2})$ is called the \emph{transpose} (or \emph{conjugate}) of $\lambda$. Consider each cell $(i,j)$ in the Young diagram of $\lambda$, the \emph{hook} of $(i,j)$ is $H_{i,j}=\{(i,j'):j'\geq j\}\cup\{(i',j):i'\geq i\}$. The \emph{hook length} of $(i,j)$ is $h_{i,j}=|H_{i,j}|$. As an important parameter, the dimension $\dim[\lambda]$ of the Specht module $S^{\lambda}$ is given by the following theorem:

\begin{theorem}(\cite{FRT1954})\label{dim_lambda}
If $\lambda$ is a partition of $n$ with hook lengths $(h_{i,j})$, then
\begin{equation}\label{eqdim_lambda}
\dim[\lambda]=\frac{n!}{\prod_{i,j}h_{i,j}}.
\end{equation}
\end{theorem}

As an immediate consequence of Theorem \ref{dim_lambda}, we have $\dim[\lambda]=\dim[\lambda^{T}]$.

\subsection{Spectra of Cayley graphs on $S_{n}$}

\subsubsection{Basics and known results}

Given a group $G$ and an inverse-closed subset $X\subseteq G$, the \emph{Cayley graph} on $G$ generated by $X$, denoted by $Cay(G,X)$, is the graph with vertex-set $G$ and edge-set $\{\{u,v\}\in {G\choose 2}: uv^{-1}\in X\}$. Cayley graphs have been studied for many years and are a class of the most important structures in algebraic graph theory. Here, we only consider a very special kind of Cayley graphs where $G=S_{n}$ and $X$ is a union of conjugacy classes.

For fixed $k\geq 1$, consider the Cayley graph $\Gamma_{k}$ on $S_n$ with generating set
$$FPF_{k}=\{\sigma\in S_{n}:\sigma \text{~has less than~}k\text{~fixed points}\}.$$
When $k=1$, the corresponding Cayley graph $\Gamma_{1}$ is also called the \emph{derangement graph} on $S_n$.

For $i,j\in[n]$, denote $\mathcal{C}_{i\rightarrow j}$ as the coset consisting of permutations $\sigma\in S_{n}$ with $\sigma(i)=j$. In \cite{EFP2011}, by taking $FPF_{k}$ as a union of conjugacy classes, the authors used the representation theory of $S_{n}$ and obtained the following results about the eigenvalues of $\Gamma_{k}$:
\begin{equation}
\lambda_{\rho}^{(k)}=\frac{1}{\dim[\rho]}\sum_{\sigma\in{FPF_k}}\chi_{\rho}(\sigma)~~~(\rho\vdash n),
\end{equation}
where the character $\chi_{\rho}$ of $\rho$ is the map defined by
\begin{align*}
\chi_{\rho} &:S_{n}\rightarrow \mathbb{C},\\
\chi_{\rho}(\sigma) &=Tr(\rho(\sigma)),
\end{align*}
and $Tr(\rho(\sigma))$ denotes the trace of the linear map of $\rho(\sigma)$. If there is no confusion, for a partition $\rho$ of $n$, we also use the notation $\rho$ to denote the corresponding \emph{irreducible representation} of $S_{n}$ (For this correspondence, see Theorem 14 in \cite{EFP2011}.).

Let $d_{n}=|FPF_{1}(n)|$ be the number of derangements in $S_{n}$, using the inclusion-exclusion formula, we have
\begin{equation*}
d_{n}=\sum_{i=0}^{n}{(-1)}^{i}{n\choose i}(n-i)!=\sum_{i=0}^{n}(-1)^{i}\frac{n!}{i!}=\left(\frac{1}{e}+o(1)\right)\cdot{n!}.
\end{equation*}
From \cite{EFP2011}, we know that for $n\geq 5$, the eigenvalues of $\Gamma_{1}$ satisfy:
\begin{align}\label{eigenvalue1}
\lambda_{(n)}^{(1)}&=d_n, \nonumber \\
\lambda_{(n-1,1)}^{(1)} &=-\frac{d_n}{(n-1)},\\
|\lambda_{\rho}^{(1)}| &<\frac{c\cdot d_{n}}{n^2}<\frac{d_{n}}{(n-1)} ~\text{for all other}~ \rho\vdash n, \nonumber
\end{align}
where $c$ is an absolute constant. And the eigenvalues of $\Gamma_{k}$ satisfy:
\begin{align}\label{eigenvalue2}
\lambda_{(n)}^{(k)}&=\sum_{i=0}^{k-1}\big[{n\choose i}\cdot d_{n-i}\big], \nonumber \\
\lambda_{(n-1,1)}^{(k)} &=\frac{1}{(n-1)}\cdot\sum_{i=0}^{k-1}\big[{n\choose i}\cdot d_{n-i}\cdot(i-1)\big],\\
|\lambda_{\rho}^{(k)}| &<\frac{c_{k}\cdot n!}{n^2} ~\text{for all other}~ \rho\vdash n, \nonumber
\end{align}
where $c_k>0$ depends on $k$ alone. As shown in \cite{EFF2017}, for $\lambda_{\rho}^{(k)}$s with different $k$s and the same $\rho$, their corresponding eigenspaces are the same $U_{\rho}$ with dimension $\dim[\rho]$. For each $t\in \mathbb{N}$, define
\begin{equation*}
U_t=\bigoplus_{\rho\vdash n: \rho_1\geq n-t}U_{\rho}.
\end{equation*}
It was proved in \cite{EFP2011} that $U_t$ is the linear span of the characteristic functions of the $t$-cosets of $S_n$, i.e.,
\begin{equation*}
U_t=Span\{\mathcal{C}_{I\rightarrow J}:I,J \text{~are ordered~} t\text{-tuples of distinct elements of~}[n]\},
\end{equation*}
where for $I=\{i_1,\ldots,i_t\}$ and $J=\{j_1,\ldots,j_t\}$, $\mathcal{C}_{I\rightarrow J}=\{\sigma\in S_n: \sigma(i_1)=j_1,\ldots,\sigma(i_t)=j_t\}$ is a $t$-coset of $S_n$. Moreover, write $V_t=\bigoplus_{\rho\vdash n: \rho_1=n-t}U_{\rho}$. Clearly, $V_t$s are pairwise orthogonal and
\begin{equation}\label{eigenspace}
U_t=U_{t-1}\bigoplus V_t.
\end{equation}


During their study of intersecting families for permutations in \cite{EFP2011}, Ellis, Friedgut and Pilpel developed several tools to estimate the spectra of $\Gamma_k$s, we include the following three lemmas which are useful for our estimations of the eigenvalues of $\Gamma_k$s.
\begin{lemma}\label{bound_eigenvalue}(\cite{EFP2011}, Lemma 6)
Let $G$ be a finite group, let $X\subseteq G$ be inverse-closed and conjugation-invariant, and let $Cay(G,X)$ be the Cayley graph on $G$ with generating set $X$. Let $\rho$ be an irreducible representation of $G$ with dimension $d$, and let $\lambda_{\rho}$ be the corresponding eigenvalue of $Cay(G,X)$. Then
\begin{equation}\label{ineqbound_eigenvalue}
|\lambda_{\rho}|\leq\frac{\sqrt{|G||X|}}{d}.
\end{equation}
\end{lemma}
\begin{lemma}\label{bound_dim1}(\cite{EFP2011}, Claim 1 in Section 3.2.1)
Let $[\rho]$ be an irreducible representation whose first row or column is of length $n-t$. Then
\begin{equation}\label{ineqbound_dim1}
\dim{[\rho]}\geq {n\choose t}e^{-t}.
\end{equation}
\end{lemma}
\begin{theorem}\label{bound_dim2}(\cite{Mishchenko07})
If $\alpha,\epsilon>0$, then there exists $N(\alpha,\epsilon)\in \mathbb{N}$ such that for all $n>N(\alpha,\epsilon)$, any irreducible representation $[\lambda]$ of $S_n$ which has all rows and columns of length at most $\frac{n}{\alpha}$ has
\begin{equation}\label{ineqbound_dim2}
\dim{[\lambda]}\geq (\alpha-\epsilon)^{n}.
\end{equation}
\end{theorem}

The above three lemmas provide a way to control $|\lambda_{\rho}^{(k)}|$ based on the dimension of $[\rho]$. When the structure of the partition $\rho$ is relatively simple, $[\rho]$'s dimension can be well bounded and therefore leads to a good control of $|\lambda_{\rho}^{(k)}|$. When the dimension of $[\rho]$ is relatively large, this method no longer works. Thus, we need the following results from \cite{KLW2015} and \cite{KLW2017}.

\begin{theorem}\label{eigenvalue3}(\cite{KLW2015}, Theorem 3.7)
Let $0<k<n$ and $\rho\vdash n$. Let $\mu_1,\ldots,\mu_q$ be the Young diagram obtained from $\rho$ by removing the right most box from any row of the diagram so that the resulting diagram is still a partition of $(n-1)$. Then
\begin{equation}\label{ineq_eigenvalue3}
\lambda_{\rho}^{(k)}=\frac{n}{k\dim[\rho]}\sum_{j=1}^{q}\dim[\mu_j]\lambda_{\mu_j}^{(k-1)}.
\end{equation}
\end{theorem}
\begin{theorem}\label{eivenvalue4}(\cite{KLW2017}, Theorem 3.5)
Let $n,k$ be integers with $n>k\geq 0$, and $\rho=(n)\vdash n$. Then
\begin{equation}\label{ineq_eigenvalue4}
\lambda_{\rho^{T}}^{(k)}={n\choose k}(-1)^{n-k-1}(n-k-1).
\end{equation}
\end{theorem}

For positive integers $n_1$ and $n_2$, we write $n_1=O_{t}(n_2)$ if $n_1\leq C_tn_2$ for some constant $C_t$ that depends only on $t$.
\begin{theorem}\label{eigenvalue5}(\cite{KLW2017}, Theorem 3.9)
Let $n,k,t$ be integers with $k\geq 0$, $t>0$ and $n>k+2t$, $\rho=(n-t,\rho_2,\ldots,\rho_r)\vdash n$ with $\sum_{i=2}^{r}\rho_i=t$, and $\beta=(\rho_2,\ldots,\rho_r)\vdash t$. Then
\begin{equation}\label{ineq_eigenvalue5}
\dim[\rho]\lambda_{\rho}^{(k)}=\dim[\beta]{n\choose k}\left(\sum_{r=0}^{t}{k\choose r}\frac{(-1)^{t-r}}{(t-r)!}\right)d_{n-k}+O_{t}(n^{2t-1+k}).
\end{equation}
\end{theorem}
Let $n,k,t$ be integers with $0\leq k<n$ and $0\le 2t<n$. We define $V(n,t)=\{\rho\vdash n: \rho=(n-t,\rho_2,\ldots,\rho_l)\text{ with }\sum_{i=2}^{l}\rho_i=t\}$, and we also need the following lemma.
\begin{lemma}\label{eigenvalue6}(\cite{KLW2017}, Lemma 3.16)
Let $t\geq 0$ and $\rho\in V(n,t)$. Then
\begin{equation}\label{ineq_eigenvalue6}
\lambda_{\rho^{T}}^{(k)}=O_{t}(n^{k+1}).
\end{equation}
\end{lemma}

\subsubsection{Some new results about $\lambda_{\rho}^{(k)}$s}

In this part, first, we shall prove two identities of the linear combinations of $\lambda_{(n)}^{(k)}$s. Then, using aforementioned results, we will provide some new estimations about $|\lambda_{\rho}^{(k)}|$ for $\rho\vdash n$ and $\rho \neq (n),(n-1,1)$.

Based on the formula of $d_{n}$, we have the following simple identity.
\begin{proposition}\label{ob1}
For any positive integer $n\geq5$, we have
\begin{equation}\label{eq_ob1}
\sum_{k=2}^{n}\sum_{i=1}^{k-1}\frac{1}{(i-1)!}\cdot\left(\sum_{s=0}^{n-i}\frac{(-1)^{s}}{s!}\right)=n-2.
\end{equation}
\end{proposition}

\begin{proof}
First, by interchanging the summation order of the LHS of (\ref{eq_ob1}), we have
\begin{align*}
\sum_{k=2}^{n}\sum_{i=1}^{k-1}\frac{1}{(i-1)!}\cdot\left(\sum_{s=0}^{n-i}\frac{(-1)^{s}}{s!}\right)&=\sum_{k=2}^{n}\sum_{i=0}^{k-2}\frac{1}{i!}\cdot\left(\sum_{s=0}^{n-i-1}\frac{(-1)^{s}}{s!}\right)\\
&=\sum_{i=0}^{n-2}\frac{1}{i!}\cdot\left(\sum_{k=i+2}^{n}\sum_{s=0}^{n-i-1}\frac{(-1)^{s}}{s!}\right)\\
&=\sum_{i=0}^{n-2}\frac{n-i-1}{i!}\cdot\left(\sum_{s=0}^{n-i-1}\frac{(-1)^{s}}{s!}\right).
\end{align*}

Now, let $a_m=\sum_{s=0}^{m}\frac{(-1)^s}{s!}$ and $A(x)$ be the generating function $\sum_{m\geq 0}a_mx^m$ of sequence $\{a_m\}_{m\geq 0}$. Then, we have $A(x)=\frac{e^{-x}}{1-x}$ and
\begin{equation*}
\sum_{m\geq 0}ma_mx^m=\left(\frac{e^{-x}}{1-x}\right)^{'}x=\frac{e^{-x}x^2}{(1-x)^2}.
\end{equation*}
Let $b_n=\sum_{i=0}^{n}\frac{n-i+1}{i!}a_{n-i+1}$ and $B(x)$ be the generating function $\sum_{n\geq 0}b_nx^n$ of sequence $\{b_n\}_{n\geq 0}$. From the above equality and the property of products of generating functions, we immediately have
\begin{equation*}
B(x)=e^{x}\cdot \frac{\sum_{m\geq 0}ma_mx^m}{x}=\frac{x}{(1-x)^2}=\sum_{n\geq 1}nx^{n}.
\end{equation*}
Therefore, $b_n=n$ and
\begin{equation*}
\sum_{k=2}^{n}\sum_{i=1}^{k-1}\frac{1}{(i-1)!}\cdot\left(\sum_{s=0}^{n-i}\frac{(-1)^{s}}{s!}\right)=\sum_{i=0}^{n-2}\frac{n-i-1}{i!}\cdot\left(\sum_{s=0}^{n-i-1}\frac{(-1)^{s}}{s!}\right)=b_{n-2}=n-2.
\end{equation*}
This completes the proof of (\ref{eq_ob1}).
\end{proof}

As an application of Proposition \ref{ob1}, we can prove the following lemma.
\begin{lemma}\label{cor1}
For any integer $n\geq 5$,
\begin{align}
&\sum_{k=1}^{n}(\lambda_{(n)}^{(k)}+(n-1)\cdot\lambda_{(n-1,1)}^{(k)})=n!\cdot(n-2), \label{id5}\\
&\sum_{k=1}^{n}(\lambda_{(n)}^{(k)}-\lambda_{(n-1,1)}^{(k)})=n!\cdot\left(n-\frac{n-2}{n-1}\right).\label{id6}
\end{align}
\end{lemma}

\begin{proof}

By (\ref{eigenvalue2}), we have
\begin{align*}
\lambda_{(n)}^{(k)}+(n-1)\cdot\lambda_{(n-1,1)}^{(k)}&=\sum_{i=0}^{k-1}{n\choose i}\cdot i\cdot d_{n-i}\\
&=0+n!\cdot \sum_{i=1}^{k-1}\frac{1}{(i-1)!}\cdot(\sum_{s=0}^{n-i}\frac{(-1)^{s}}{s!}),
\end{align*}
where the second term $n!\cdot\sum_{i=1}^{k-1}\frac{1}{(i-1)!}\cdot(\sum_{s=0}^{n-i}\frac{(-1)^{s}}{s!})$ in the RHS of the above equality equals $0$ when $k\leq 1$. Thus the identity (\ref{id5}) follows from Lemma \ref{ob1}.

Similarly, by (\ref{eigenvalue2}), we have
\begin{align*}
\lambda_{(n)}^{(k)}-\lambda_{(n-1,1)}^{(k)}&=\sum_{i=0}^{k-1}{n\choose i}\cdot \left(1-\frac{i-1}{n-1}\right)\cdot d_{n-i}\\
&=\frac{n}{n-1}\cdot \sum_{i=0}^{k-1}{n\choose i}\cdot d_{n-i}-\frac{1}{n-1}\cdot \sum_{i=0}^{k-1}{n\choose i}\cdot i\cdot d_{n-i}.
\end{align*}
Therefore,
\begin{align*}
\sum_{k=1}^{n}(\lambda_{(n)}^{(k)}-\lambda_{(n-1,1)}^{(k)})&=\sum_{k=1}^{n}\left(\frac{n}{n-1}\cdot \sum_{i=0}^{k-1}{n\choose i}\cdot d_{n-i}-\frac{1}{n-1}\cdot \sum_{i=0}^{k-1}{n\choose i}\cdot i\cdot d_{n-i}\right)\\
&=n!\cdot\frac{n}{n-1}\cdot\sum_{k=1}^{n}\sum_{i=0}^{k-1}\frac{1}{i!}\cdot \left(\sum_{s=0}^{n-i}\frac{(-1)^{s}}{s!}\right)-n!\cdot\frac{n-2}{n-1}.
\end{align*}
From the definition of $M_n$ in Proposition \ref{ob1}, we have $\sum_{k=1}^{n}\sum_{i=0}^{k-1}\frac{1}{i!}\cdot (\sum_{s=0}^{n-i}\frac{(-1)^{s}}{s!})=M_n$. Thus we have
\begin{align*}
\sum_{k=1}^{n}(\lambda_{(n)}^{(k)}-\lambda_{(n-1,1)}^{(k)})=n!\cdot\left(n-\frac{n-2}{n-1}\right).
\end{align*}
This completes the proof.
\end{proof}

Denote $\Phi=\{\rho\vdash n:\rho\neq (n),(n-1,1)\}$. According to (\ref{eigenvalue2}), for fixed $k$ and $\rho\neq (n),(n-1,1)$, we have $|\lambda_{\rho}^{(k)}|\leq \frac{c_{k}\cdot n!}{n^2}$. However, this bound is not good enough. When the index $k$ varies from $1$ to $n$, the constant $c_k$ might become relatively large. Thus, if we try to get similar identities as (\ref{id5}) and (\ref{id6}) for $\lambda_{\rho}^{(k)}$ with $\rho\in \Phi$, we need some more delicate evaluations about $\lambda_{\rho}^{(k)}$s for $\rho\in \Phi$.

According to the structure of their corresponding partitions, we can divide $\rho\in \Phi$ into the following four parts:
\begin{align*}
\Phi_1=&\{\rho\in {\Phi}: \text{the first row or column of } \rho \text{ is of length at most } n-3\};\\
\Phi_2=&\{(n)^{T}\};\\
\Phi_3=&\{(n-2,1,1),(n-2,2)\};\\
\Phi_4=&\{(n-1,1)^{T},(n-2,1,1)^{T},(n-2,2)^{T}\}.
\end{align*}
Clearly, $\Phi$ are formed by these four parts and all of them are pairwise disjoint. Based on the known results, we can prove the following bounds about $\lambda_{\rho}^{(k)}$s for $\rho\in \Phi$.

\begin{lemma}\label{lem1}
Let $n,k,t$ be positive integers with $n$ sufficiently large. Then,
\begin{itemize}
  \item When $\rho\in \Phi_1\sqcup\Phi_2$, we have $|\lambda_{\rho}^{(k)}|\leq\frac{7e^3 n!}{n^3}$ for all $1\leq k\leq n$.
  \item When $\rho\in \Phi_3$, we have $\lambda_{\rho}^{(k)}\geq-\frac{c_0n!}{n^3}$ for $3\leq k\leq n-\frac{n}{\ln{n}}-7$, where $c_0$ is an absolute constant; and $|\lambda_{\rho}^{(k)}|\leq \frac{3n!}{n^2}$ for $k=1,2$ or $k> n-\frac{n}{\ln{n}}-7$.
  \item When $\rho\in \Phi_4$, we have $|\lambda_{\rho}^{(k)}|\leq\frac{n!}{n^3}$ for $1\leq k\leq n-\frac{n}{\ln{n}}-7$; and $|\lambda_{\rho}^{(k)}|\leq \frac{3n!}{n^2}$ for $k> n-\frac{n}{\ln{n}}-7$.
\end{itemize}
\end{lemma}

\begin{proof}
Consider the eigenvalues corresponding to irreducible representations in $\Phi_1\sqcup\Phi_2$. For each $\rho\in\Phi_1$, assume that the length of the first row or column of $\rho$ is $n-t$. When $3\leq t\leq \frac{n}{3}$, since ${n\choose t}e^{-t}$ is increasing in the range $3\leq t\leq \frac{n-e}{e+1}$ and is decreasing in the range $\frac{n-e}{e+1}< t\leq \frac{n}{3}$, thus, we have ${n\choose t}e^{-t}\geq \frac{n^3}{7e^3}$. By Lemma \ref{bound_dim1}, $\dim[\rho]\geq {n\choose t}e^{-t}\geq \frac{n^3}{7e^3}$. When $t\geq \frac{n}{3}$, $\rho$ has all rows and columns of length at most $\frac{2n}{3}$. Since $n$ is sufficiently large, by Theorem \ref{bound_dim2}, $\dim[\rho]\geq (\frac{3}{2}-\epsilon)^{n}\geq \frac{n^3}{7e^3}$. Therefore, for all $\rho\in\Phi_1$, we have $\dim[\rho]\geq \frac{n^3}{7e^3}$. Note that $|FPF_{k}|< n!$. By Lemma \ref{bound_eigenvalue}, we have
\begin{equation*}
|\lambda_{\rho}^{(k)}|\leq \frac{7e^3 n!}{n^3}
\end{equation*}
for all $\rho\in\Phi_1$ and $1\leq k\leq n$. According to Theorem \ref{eivenvalue4}, $\lambda_{(n)^{T}}^{(k)}=(-1)^{n-k-1}(n-k-1){n\choose k}$. Thus, we also have $|\lambda_{(n)^{T}}^{(k)}|\leq \frac{7e^3 n!}{n^3}$.


Consider the eigenvalues corresponding to irreducible representations in $\Phi_3$. Based on structures of Young diagrams of $(n-2,1,1)$ and $(n-2,2)$, one can easily get their hook lengths. Thus, by Theorem \ref{dim_lambda}, $\dim[(n-2,1,1)]=\frac{(n-1)(n-2)}{2}$ and $\dim[(n-2,2)]=\frac{(n-1)(n-3)}{2}$. Take $t=2$ in Theorem \ref{eigenvalue5}, for $1\leq k\leq n-5$, we have
\begin{align*}
\frac{(n-1)(n-2)}{2}\lambda_{(n-2,1,1)}^{(k)}&={n\choose k}\frac{k^2-3k+1}{2}d_{n-k}+O_{2}(n^{k+3});\\
\frac{(n-1)(n-3)}{2}\lambda_{(n-2,2)}^{(k)}&={n\choose k}\frac{k^2-3k+1}{2}d_{n-k}+O_{2}(n^{k+3}).
\end{align*}
For $k=1,2$ and $n$ sufficiently large, this leads to $\lambda_{(n-2,1,1)}^{(1)}=\lambda_{(n-2,2)}^{(1)}=-\left(\frac{1}{e}+o(1)\right)\cdot\frac{n!}{n^2}$ and $\lambda_{(n-2,1,1)}^{(2)}=\lambda_{(n-2,2)}^{(2)}=-\left(\frac{1}{2e}+o(1)\right)\cdot\frac{n!}{n^2}$. For $3\leq k\leq n-\frac{n}{\ln{n}}-7$, we have ${n\choose k}\frac{k^2-3k+1}{2}d_{n-k}>0$ and $n^{k+3}<\frac{n!}{n^3}$. This indicates that
\begin{align*}
\lambda_{(n-2,1,1)}^{(k)}\geq -c_1\frac{n!}{n^3} \text{~and~} \lambda_{(n-2,2)}^{(k)}\geq -c_2\frac{n!}{n^3}
\end{align*}
for all $3\leq k\leq n-\frac{n}{\ln{n}}-7$, where $c_1,c_2\geq 0$ are absolute constants. For $k> n-\frac{n}{\ln{n}}-7$, since we already have $\dim[(n-2,1,1)]=\frac{(n-1)(n-2)}{2}$ and $\dim[(n-2,2)]=\frac{(n-1)(n-3)}{2}$, by Lemma \ref{bound_eigenvalue} and $|FPF_{k}|< n!$, we have
\begin{align*}
|\lambda_{(n-2,1,1)}^{(k)}|,~|\lambda_{(n-2,2)}^{(k)}|\leq \frac{3n!}{n^2}.
\end{align*}

Consider the eigenvalues corresponding to irreducible representations in $\Phi_4$. For $1\leq k\leq n-\frac{n}{\ln{n}}-7$, by Lemma \ref{eigenvalue6}, we have $\lambda_{(n-1,1)^{T}}^{(k)}=O_{1}(n^{k+1})$ and $\lambda_{(n-2,1,1)^{T}}^{(k)}=\lambda_{(n-2,2)^{T}}^{(k)}=O_2(n^{k+1})$. Therefore, we have
\begin{align*}
|\lambda_{(n-1,1)^{T}}^{(k)}|,~|\lambda_{(n-2,1,1)^{T}}^{(k)}|,~|\lambda_{(n-2,2)^{T}}^{(k)}|\leq \frac{n!}{n^3},
\end{align*}
for all $1\leq k\leq n-\frac{n}{\ln{n}}-7$. For $k> n-\frac{n}{\ln{n}}-7$, based on the structure of $(n-1,1)^{T}$, we have
\begin{equation*}
\lambda_{(n-1,1)^{T}}^{(k)}=\frac{n}{k\dim[\rho]}\cdot(\dim[(n-1)^{T}]\cdot\lambda_{(n-1)^{T}}^{(k-1)}+\dim[(n-2,1)^{T}]\cdot\lambda_{(n-2,1)^{T}}^{(k-1)})
\end{equation*}
by Theorem \ref{eigenvalue3}. Since $\dim[(n-1)^{T}]=\dim[(n-1)]=1$ and $\dim[\rho]=\dim[(n-1,1)]=n-1$, we further have
\begin{equation*}
|\lambda_{(n-1,1)^{T}}^{(k)}|\leq\frac{n}{k(n-1)}\cdot|\lambda_{(n-1)^{T}}^{(k-1)}|+\frac{n}{k}\cdot|\lambda_{(n-2,1)^{T}}^{(k-1)}|.
\end{equation*}
From the first part of this proof, $|\lambda_{(n-1)^{T}}^{(k-1)}|\leq\frac{7e^3 (n-1)!}{(n-1)^3}<\frac{n!}{2n^2}$. Meanwhile, since $\dim[(n-2,1)^{T}]=n-2$, by Lemma \ref{bound_eigenvalue}, we have $|\lambda_{(n-2,1)^{T}}^{(k-1)}|\leq \frac{(n-1)!}{n-2}<\frac{3n!}{2n^2}$. Therefore, by the choice of $k$, we have
\begin{equation*}
|\lambda_{(n-1,1)^{T}}^{(k)}|\leq \frac{3n!}{n^2}.
\end{equation*}
Similarly, note that $\dim[(n-2,1,1)^{T}]=\dim[(n-2,1,1)]$ and $\dim[(n-2,2)^{T}]=\dim[(n-2,2)]$, by Lemma \ref{bound_eigenvalue}, we also have
\begin{align*}
|\lambda_{(n-2,1,1)^{T}}^{(k)}|,~|\lambda_{(n-2,2)^{T}}^{(k)}|\leq \frac{3n!}{n^2}.
\end{align*}
This completes the proof.
\end{proof}

\section{Proof of Theorem \ref{stabilityforsub}}

Let $n$ be a positive integer and $V$ be an $n$-dimensional vector space over $\mathbb{F}_q$. In the following, if there is no confusion, we shall omit the field size $q$ in the \emph{Gaussian binomial coefficient} and use ``dim'' in short for ``dimensional''.
\begin{lemma}\label{lem_sub1}\cite{GC2016}
Let $\mathbf{\alpha}$ be a $k$-dim subspace of $V$. Then, for integers $j,l$ satisfying $0\leq j\leq l$, the number of $l$-dim subspaces of $V$ whose intersection with $\mathbf{\alpha}$ has dimension $j$ is
\begin{align*}
q^{(k-j)(l-j)}\Gaussbinom{n-k}{l-j}\Gaussbinom{k}{j}.
\end{align*}
\end{lemma}


\begin{proposition}\label{counting intersection subspace}
For integer $1\leq t\leq n$, denote $U_0$ as a $t$-dim subspace of $V$. Let $\mathcal{F}$ be the family of all $k$-dim subspaces of $V$ containing $U_0$. Then, we have $|\mathcal{F}|=\Gaussbinom{n-t}{k-t}$ and
\begin{align}\label{sub_ineq01}
\mathcal{I}(\mathcal{F})=\left(\sum_{j=0}^{k-t}(j+t)q^{(k-t-j)^2}\Gaussbinom{n-k}{k-t-j}\Gaussbinom{k-t}{j}\right)\Gaussbinom{n-t}{k-t}.
\end{align}
\end{proposition}
\begin{proof}
The first statement is an immediate consequence of Lemma \ref{lem_sub1}.

Denote $V=U_0\oplus V_1$ and take $\mathcal{G}_0$ as the family of all $(k-t)$-dim subspaces of $V_1$. Therefore, $\mathcal{F}=U_0\oplus \mathcal{G}_0=\{U_0\oplus G:G\in \mathcal{G}_0\}$. For $F_1\in\mathcal{F}$, let $F_1=U_0\oplus G_1$. Then, we have $\mathcal{I}(F_1,\mathcal{F})=\sum_{F\in\mathcal{F}}|\dim(F_1\cap F)|=\sum_{G\in \mathcal{G}_0}(|\dim(G_1\cap G)|+t)$. Combined with Lemma \ref{lem_sub1}, this leads to
\begin{align*}
\mathcal{I}(F_1,\mathcal{F})=\sum_{j=0}^{k-t}(j+t)q^{(k-t-j)^2}\Gaussbinom{(n-t)-(k-t)}{k-t-j}\Gaussbinom{k-t}{j}.
\end{align*}
Therefore, (\ref{sub_ineq01}) follows from $\mathcal{I}(\mathcal{F})=\sum_{F\in\mathcal{F}}\mathcal{I}(F,\mathcal{F})$.
\end{proof}

Now, we present the proof of Theorem \ref{stabilityforsub}.

\begin{proof}[Proof of Theorem \ref{stabilityforsub}]
First, we shall show that the number of popular $t$-dim subspaces is not large.
\begin{claim}\label{claim1}
Let $\mathcal{X}=\{U\in \Gaussbinom{V}{t} : |\mathcal{F}(U)|\ge \frac{|\mathcal{F}|}{(2k+2)\Gaussbinom{k}{t}}\}$, then $|\mathcal{X}|<(4k+4)\Gaussbinom{k}{t}$.
\end{claim}

\begin{proof}
Otherwise, assume that there is an $\mathcal{X}_0\subseteq \mathcal{X}$ such that $|\mathcal{X}_0|=(4k+4)\Gaussbinom{k}{t}$. We have
\begin{align*}
|\mathcal{F}|\geq |\bigcup\limits_{U\in \mathcal{X}}\mathcal{F}(U)|&\geq \sum_{U\in \mathcal{X}_0}|\mathcal{F}(U)|-\sum_{U_1\neq U_2\in \mathcal{X}_0}|\mathcal{F}(U_1+ U_2)|\\
&\geq 2|\mathcal{F}|-{|\mathcal{X}_0|\choose 2}\Gaussbinom{n-t-1}{k-t-1}.
\end{align*}
Since $|\mathcal{F}|=\delta\Gaussbinom{n-t}{k-t}$ and $\delta\geq \frac{(4k+4)^2n}{q^{n-k}}$, based on the choice of $n$ and $\delta$, we have
\begin{align}\label{sub_ineq01.5}
|\mathcal{F}|&\geq \frac{(4k+4)^2n}{q^{n-k}}\cdot\Gaussbinom{n-t}{k-t}=\frac{(4k+4)^2n}{q^{n-k}}\cdot \frac{q^{n-t}-1}{q^{k-t}-1}\cdot\Gaussbinom{n-t-1}{k-t-1}\nonumber\\
&= (4k+4)^2n\cdot \frac{q^{n-t}-1}{q^{n-t}-q^{n-k}}\cdot\Gaussbinom{n-t-1}{k-t-1}\nonumber\\
&> 8(k+1)^2\cdot{\Gaussbinom{k}{t}}^2\cdot\Gaussbinom{n-t-1}{k-t-1}\geq {|\mathcal{X}_0|\choose 2}\Gaussbinom{n-t-1}{k-t-1}.
\end{align}
This leads to $2|\mathcal{F}|-{|\mathcal{X}_0|\choose 2}\Gaussbinom{n-t-1}{k-t-1}>|\mathcal{F}|$, a contradiction.
\end{proof}

Claim \ref{claim1} enables us to proceed further estimation on $\mathcal{I}(\mathcal{F})$. Next, we shall prove that the most popular $t$-dim subspace is contained in the majority members of $\mathcal{F}$.

\begin{claim}\label{claim2}
There exists a $t$-dim subspace $U_0\in \mathcal{X}$ such that $|\mathcal{F}(U_0)|\ge (1-\frac{2}{3k+3})|\mathcal{F}|$.
\end{claim}
\begin{proof}
Denote $U_0$ as the most popular $t$-dim subspace appearing in the members of $\mathcal{F}$.
\begin{itemize}
  \item When $\delta\leq 1$.
\end{itemize}

Consider the new family $\mathcal{F}_0\subseteq \Gaussbinom{V}{k}$ of size $|\mathcal{F}|$ and $U_0\subseteq F$ for all $F\in \mathcal{F}_0$. According to (\ref{basic_id}), we have $\mathcal{I}(\mathcal{F}_0)\geq t|\mathcal{F}|^2$. Therefore, by the optimality of $\mathcal{F}$, $\mathcal{I}(\mathcal{F})\geq t|\mathcal{F}|^2$.

Given a positive integer $q$, for variable $x\in \mathbb{R}^{+}$, define the function $\Gaussbinom{x}{k}=\prod\limits_{i=0}^{k-1}\frac{q^{x-i}-1}{q^{k-i}-1}$. One can easily verify that $\Gaussbinom{x}{k}$ is a convex increasing function when $x\geq k-1$. Thus by Jensen Inequality, we have
\begin{align}\label{sub_ineq02}
\sum_{U\in\Gaussbinom{V}{t}}|\mathcal{F}(U)|^2=\sum_{A,B\in\mathcal{F}}\Gaussbinom{\dim(A\cap B)}{t}&\ge \Gaussbinom{\frac{\sum_{A,B\in\mathcal{F}}\dim(A\cap B)}{|\mathcal{F}|^2}}{t}\cdot|\mathcal{F}|^2.
\end{align}
Note that $\mathcal{I}(\mathcal{F})=\sum_{A,B\in\mathcal{F}}\dim(A\cap B)$, (\ref{sub_ineq02}) leads to
$\sum_{U\in\Gaussbinom{V}{t}}|\mathcal{F}(U)|^2\geq \Gaussbinom{\frac{\mathcal{I}(\mathcal{F})}{|\mathcal{F}|^2}}{t}\cdot|\mathcal{F}|^2\geq |\mathcal{F}|^2$. Moreover, we also have
\begin{align*}
\sum_{U\in\Gaussbinom{V}{t}}|\mathcal{F}(U)|^2&=\sum_{U\in \mathcal{X}}|\mathcal{F}(U)|^2+\sum_{U\notin \mathcal{X}}|\mathcal{F}(U)|^2 \le|\mathcal{F}(U_0)|\cdot\sum_{U\in \mathcal{X}}|\mathcal{F}(U)|+\frac{|\mathcal{F}|}{(2k+2)\Gaussbinom{k}{t}}\cdot\sum_{U\notin \mathcal{X}}|\mathcal{F}(U)|\nonumber\\
&\leq |\mathcal{F}(U_0)|\cdot\left(|\mathcal{F}|+\sum_{U_1\neq U_2\in \mathcal{X}}|\mathcal{F}(U_1+ U_2)|\right)+\frac{|\mathcal{F}|}{(2k+2)\Gaussbinom{k}{t}}\cdot\sum_{U\in \Gaussbinom{V}{t}}|\mathcal{F}(U)|.
\end{align*}
Note that $\dim(U_1+ U_2)\geq t+1$ for $U_1\neq U_2\in \mathcal{X}$ and $\sum_{U\in \Gaussbinom{V}{t}}|\mathcal{F}(U)|=\sum_{F\in \mathcal{F}}|\{U\subseteq F: U\in \Gaussbinom{V}{t}\}|= |\mathcal{F}|\Gaussbinom{k}{t}$, we can further obtain
\begin{align}\label{sub_ineq02.5}
\sum_{U\in\Gaussbinom{V}{t}}|\mathcal{F}(U)|^2&\leq |\mathcal{F}(U_0)|\cdot\left(|\mathcal{F}|+{|\mathcal{X}|\choose 2}\Gaussbinom{n-t-1}{k-t-1}\right)+\frac{|\mathcal{F}|}{(2k+2)\Gaussbinom{k}{t}}\cdot\sum_{F\in\mathcal{F}}|\{U\subseteq F: U\in \Gaussbinom{V}{t}\}|\nonumber\\
&\leq |\mathcal{F}(U_0)|\cdot\left(|\mathcal{F}|+{|\mathcal{X}|\choose 2}\Gaussbinom{n-t-1}{k-t-1}\right)+\frac{|\mathcal{F}|^2}{2k+2}.
\end{align}
According to the calculation of (\ref{sub_ineq01.5}), the choice of $\delta$ leads to ${|\mathcal{X}|\choose 2}\Gaussbinom{n-t-1}{k-t-1}\leq \frac{{\Gaussbinom{k}{t}}^{2}|\mathcal{F}|}{n}$. Note that $n\geq (4k+4)^2\Gaussbinom{k}{t}^2$, this indicates that ${|\mathcal{X}|\choose 2}\Gaussbinom{n-t-1}{k-t-1}\leq\frac{|\mathcal{F}|}{(4k+4)^2}$. Therefore, by (\ref{sub_ineq02.5}), we have $|\mathcal{F}(U_0)|\geq (1-\frac{2}{3k+3})|\mathcal{F}|$.

\begin{itemize}
  \item When $\delta> 1$.
\end{itemize}

Write $U_0=U_1\oplus \langle u_0\rangle$, where $U_1$ is a $(t-1)$-dim subspace of $U_0$ and $\langle u_0\rangle$ is the $1$-dim subspace spanned by some $u_0\in U_0$. Let $U'=U_1\oplus \langle u_1\rangle$ be another $t$-dim subspace of $V$, where $u_1\in V\setminus U_0$. Consider the new family $\mathcal{G}_0=\mathcal{G}_1\sqcup \mathcal{G}_2$ with size $|\mathcal{F}|$, where $\mathcal{G}_1$ consists of all $k$-dim subspaces containing $U_0$ and $\mathcal{G}_2$ consists of $(\delta-1)\Gaussbinom{n-t}{k-t}$ $k$-dim subspaces containing $U'$. Based on the structure of $\mathcal{G}_0$, according to (\ref{basic_id}), we have
\begin{align*}
\mathcal{I}(\mathcal{G}_0)&\geq (t-1)|\mathcal{F}|^2+|\mathcal{G}_1|^2+|\mathcal{G}_2|^2\\
&= t|\mathcal{F}|^2-2|\mathcal{G}_1||\mathcal{G}_2|=(t-\frac{2(\delta-1)}{\delta^2})|\mathcal{F}|^2.
\end{align*}

Again, by the optimality of $\mathcal{F}$, we have $\mathcal{I}(\mathcal{F})\geq \mathcal{I}(\mathcal{G}_0)\geq (t-\frac{2(\delta-1)}{\delta^2})|\mathcal{F}|^2$. Therefore, (\ref{sub_ineq02}) leads to
$\sum_{U\in\Gaussbinom{V}{t}}|\mathcal{F}(U)|^2\geq \Gaussbinom{t-\frac{2(\delta-1)}{\delta^2}}{t}\cdot|\mathcal{F}|^2$. Now, consider the function $\Gaussbinom{t-x}{t}=\prod_{i=0}^{t-1}\frac{q^{t-x-i}-1}{q^{t-i}-1}$ for $x\in \mathbb{R}$ satisfying $0<x<1$. Clearly, $\Gaussbinom{t-x}{t}$ is a decreasing function and when $x$ is fixed, the term $\frac{q^{t-x-i}-1}{q^{t-i}-1}$ is decreasing as $i$ increases. Therefore, we have
\begin{align*}
\Gaussbinom{t-x}{t}&=\prod_{i=0}^{t-1}\frac{q^{t-x-i}-1}{q^{t-i}-1}\geq (\frac{q^{1-x}-1}{q-1})^{t}.
\end{align*}
Since $\delta\leq 1+\frac{1}{96 t(k+1)\ln{q}}$, we have $\frac{2(\delta-1)}{\delta^2}\leq\frac{1}{48 t(k+1)\ln{q}}$. Denote $\varepsilon=\frac{1}{48 t(k+1)\ln{q}}$. Then, we have $\Gaussbinom{t-\frac{2(\delta-1)}{\delta^2}}{t}\geq (\frac{q^{1-\varepsilon}-1}{q-1})^{t}=(1-\frac{1-q^{-\varepsilon}}{1-q^{-1}})^{t}$. Note that for $q\geq 2$, $1-q^{-\varepsilon}\leq \varepsilon\ln{q}$ and $1-\frac{1}{q}\geq \frac{1}{2}$. Thus, we have
$$\Gaussbinom{t-\frac{2(\delta-1)}{\delta^2}}{t}\geq (1-2\varepsilon\ln{q})^{t}\geq 1-2t\varepsilon\ln{q}= 1-\frac{1}{24(k+1)}.$$
This leads to
$$\sum_{U\in\Gaussbinom{V}{t}}|\mathcal{F}(U)|^2\geq \Gaussbinom{t-\frac{2(\delta-1)}{\delta^2}}{t}\cdot|\mathcal{F}|^2\geq (1-\frac{1}{24(k+1)})|\mathcal{F}|^{2}.$$
Combined with the upper bound given by (\ref{sub_ineq02.5}), by the choice of $n$ and $\delta$, we also have $|\mathcal{F}(U_0)|\geq (1-\frac{2}{3k+3})|\mathcal{F}|$.
\end{proof}

Finally, we show that when $\delta\leq 1$, $U_0$ is contained in all members of $\mathcal{F}$; when $\delta> 1$, all $k$-dim subspaces of $V$ that contains $U_0$ are in $\mathcal{F}$.

\begin{itemize}
  \item When $\delta\leq 1$.
\end{itemize}

Assume that there exists an $F_0\in\mathcal{F}$ such that $U_0\nsubseteq F_0$. Since for each $F\in \mathcal{F}$,
\begin{align}\label{sub_ineq03}
I(F,\mathcal{F})= \sum_{A\in \mathcal{F}}\dim(F\cap A)= \sum_{U_0\subseteq A, A\in \mathcal{F}}\dim(F\cap A)+\sum_{U_0\nsubseteq A, A\in\mathcal{F}}\dim(F\cap A).
\end{align}
Take $F=F_0$ in the above equality and consider the first term $\sum_{U_0\subseteq A, A\in \mathcal{F}}\dim(F_0\cap A)$ in the RHS. Assume that $A=A_0\oplus U_0$ and $F_0=F_1\oplus (U_0\cap F_0)$. When $\dim(F_0\cap A)\ge t$, knowing that $U_0\nsubseteq F_0$, we have $|\dim(A_0\cap F_1)|\geq 1$. Therefore, we can write $A_0=A_1\oplus U_1$ for some $1$-dim subspace in $F_1$. Note that there are at most $\Gaussbinom{k}{1}$ different choices of such $U_1\subseteq F_1$. And for each fixed $U_1$, there are at most $\Gaussbinom{n-(t+1)}{k-(t+1)}$ different choices of $A$ satisfying $U_0\oplus U_1\subseteq A$. Therefore, the number of such $A$s is at most $\Gaussbinom{k}{1}\Gaussbinom{n-t-1}{k-t-1}$. When $\dim(F_0\cap A)\le t-1$, since $A\in \mathcal{F}$, the number of such $A$s is upper bounded by $|\mathcal{F}(U_0)|$. Therefore, we have
\begin{equation*}
\sum_{U_0\subseteq A, A\in \mathcal{F}}\dim(F_0\cap A)\leq(k-1)\Gaussbinom{k}{1}\Gaussbinom{n-t-1}{k-t-1}+(t-1)|\mathcal{F}(U_0)|.
\end{equation*}
As for the second term, we have that $\sum_{U_0\nsubseteq A, A\in \mathcal{F}}\dim(F_0\cap A)\leq k(|\mathcal{F}|-|\mathcal{F}(U_0)|)$. Therefore, combined with Claim \ref{claim2}, this leads to
\begin{align}\label{sub_ineq04}
I(F_0,\mathcal{F})&\leq k\left(|\mathcal{F}|+\Gaussbinom{k}{1}\Gaussbinom{n-t-1}{k-t-1}\right)-(k-t+1)|\mathcal{F}(U_0)|\nonumber\\
&\leq (t-\frac{2t+k+1}{3k+3})|\mathcal{F}|+k\Gaussbinom{k}{1}\Gaussbinom{n-t-1}{k-t-1}.
\end{align}

From the assumption, we know that $\mathcal{F}$ is not contained in any full $t$-star. Therefore, we can replace $F_0$ with some $F'\notin \mathcal{F}$ containing $U_0$. Denote the resulting new family as $\mathcal{F}'$, we have
\begin{align*}
\mathcal{I}(\mathcal{F}')-\mathcal{I}(\mathcal{F})&=\sum_{F\in\mathcal{F}'}\mathcal{I}(F,\mathcal{F}')-\sum_{F\in \mathcal{F}}\mathcal{I}(F,\mathcal{F})\\
&=2(\mathcal{I}(F',\mathcal{F}\setminus\{F_0\})-\mathcal{I}(F_0,\mathcal{F}\setminus\{F_0\})),
\end{align*}
where the second equality follows from $\mathcal{F}'\setminus\{F'\}=\mathcal{F}\setminus\{F_0\}$. By (\ref{sub_ineq03}) and Claim \ref{claim2}, we have $I(F',\mathcal{F}\setminus\{F_0\})\geq t|\mathcal{F}(U_0)|\geq (t-\frac{2t}{3k+3})|\mathcal{F}|$. Therefore, based on (\ref{sub_ineq04}) and the calculations in (\ref{sub_ineq01.5}), we have
\begin{align*}
\mathcal{I}(F',\mathcal{F}\setminus\{F_0\})-\mathcal{I}(F_0,\mathcal{F}\setminus\{F_0\})&\geq \mathcal{I}(F',\mathcal{F}\setminus\{F_0\})-\mathcal{I}(F_0,\mathcal{F})\\
&\geq (t-\frac{2t}{3k+3})|\mathcal{F}|-(t-\frac{2t+k+1}{3k+3})|\mathcal{F}|-k\Gaussbinom{k}{1}\Gaussbinom{n-t-1}{k-t-1}\\
&\geq \frac{|\mathcal{F}|}{3}-\frac{k\Gaussbinom{k}{1}}{8(k+1)^2\Gaussbinom{k}{t}^{2}}|\mathcal{F}|\\
&\geq (\frac{1}{3}-\frac{1}{8(k+1)\Gaussbinom{k}{t}})|\mathcal{F}|>0.
\end{align*}
This contradicts the fact that $\mathcal{I}(\mathcal{F})=\mathcal{MI}(\mathcal{F})$. Thus, all $F\in \mathcal{F}$ must contain $U_0$.

\begin{itemize}
  \item When $\delta> 1$.
\end{itemize}

Assume that there exists a $G'\in \Gaussbinom{V}{k}\setminus\mathcal{F}$ with $U_0\subseteq G'$. Since $|\mathcal{F}|=\delta\Gaussbinom{n-t}{k-t}$ and $\delta>1$, clearly, there exists some $G_0\in\mathcal{F}$ such that $U_0\nsubseteq G_0$. Take $F=G_0$ in (\ref{sub_ineq03}), since the estimation in (\ref{sub_ineq04}) is irrelevant to the choice of $\delta$. Thus, with similar procedures, we can also obtain $\mathcal{I}(G_0,\mathcal{F})\leq (t-\frac{2t+k+1}{3k+3})|\mathcal{F}|+k\Gaussbinom{k}{1}\Gaussbinom{n-t-1}{k-t-1}$. On the other hand, by (\ref{sub_ineq03}), we also have $\mathcal{I}(G',\mathcal{F}\setminus\{G_0\})\geq t|\mathcal{F}(U_0)|\geq(t-\frac{2t}{3k+3})|\mathcal{F}|$. Again, we can replace $G_0$ with $G'$ and denote the resulting new family as $\mathcal{F}'$. With similar arguments as those for the case $\delta\leq 1$, this procedure increases the value of $\mathcal{I}(\mathcal{F})$ strictly, a contradiction. Therefore, all $k$-dim subspaces of $V$ containing $U_0$ are in $\mathcal{F}$.

This completes the proof of Theorem \ref{stabilityforsub}.
\end{proof}

%

\section{Proof of Theorem \ref{theremovallemma}}

For any integer $s\geq\frac{1}{2}(n-1)!$, there exist unique $k\in\mathbb{N}$ and $\varepsilon\in(-\frac{1}{2},\frac{1}{2}]$ such that $s=(k+\varepsilon)(n-1)!$. Denote $\mathcal{T}_{0}(n,s)$ as the subfamily of $\mathcal{T}(n,s)$ consisting of $\lfloor k+\varepsilon\rfloor=a_0$ pairwise disjoint $1$-cosets and $\lfloor(k+\varepsilon-a_0)(n-1)\rfloor=a_1$ pairwise disjoint $2$-cosets from another $1$-coset disjoint from the former $a_0$ $1$-cosets.


Assume that
\begin{align}\label{ineq01}
\mathcal{T}_0(n,s)=(\bigsqcup_{i=2}^{a_1+1}\mathcal{C}_{{1\rightarrow 1},{2\rightarrow i}})\sqcup(\bigsqcup_{j=2}^{a_0+1} \mathcal{C}_{{1\rightarrow j}}),
\end{align}
where $\mathcal{C}_{{1\rightarrow 1},{2\rightarrow i}}=\{\sigma\in S_n: \sigma(1)=1 \text{~and~} \sigma(2)=i\}$ and $\mathcal{C}_{{1\rightarrow j}}=\{\sigma\in S_n: \sigma(1)=j\}$. 
Note that for every $\mathcal{T}\subseteq S_n$,
\begin{align*}
\mathcal{I}(\mathcal{T})&=\sum_{i,j\in [n]}|\mathcal{T}_{{i\rightarrow j}}|^2,
\end{align*}
where $\mathcal{T}_{{i\rightarrow j}}=\{\sigma\in \mathcal{T}: \sigma(i)=j\}$. Hence, when $0\leq a_0\leq a_1\leq n-1$, we have
\begin{align}\label{ineq02}
\mathcal{I}(\mathcal{T}_{0}(n,s))&=\sum_{i,j\in [n]}|\mathcal{T}_{0}(n,s)_{i\rightarrow j}|^{2}=\sum_{j\in [n]}|\mathcal{T}_{0}(n,s)_{1\rightarrow j}|^{2}+\sum_{j\in [n]}|\mathcal{T}_{0}(n,s)_{2\rightarrow j}|^{2}+\sum_{i\in [3,n]}\sum_{j\in [n]}|\mathcal{T}_{0}(n,s)_{i\rightarrow j}|^{2}\nonumber\\
&=\big[(a_1(n-2)!)^2+a_0((n-1)!)^2\big]+((n-2)!)^2(a_0^2 n+2a_0a_1-2a_0^2+a_1-a_0)+\nonumber\\
&~~~~(n-2)\big[(a_0(n-2)!)^2+a_0((a_0-1)(n-2)!+(a_1-1)(n-3)!)^2+\nonumber\\
&~~~~~~~~~~~~~~~(a_1-a_0)(a_0(n-2)!+(a_1-1)(n-3)!)^2+\nonumber\\
&~~~~~~~~~~~~~~~(n-a_1-1)(a_0(n-2)!+a_1(n-3)!)^2\big].
\end{align}
When $0\leq a_1\leq a_0\leq n-1$, similarly, we have
\begin{align}\label{ineq02.5}
\mathcal{I}(\mathcal{T}_{0}(n,s))&=\big[(a_1(n-2)!)^2+a_0((n-1)!)^2\big]+((n-2)!)^2(a_0^2 n+2a_0a_1-2a_0^2+a_0-a_1)+\nonumber\\
&~~~~(n-2)\big[(a_0(n-2)!)^2+a_1((a_0-1)(n-2)!+(a_1-1)(n-3)!)^2+\nonumber\\
&~~~~~~~~~~~~~~~(a_0-a_1)((a_0-1)(n-2)!+a_1(n-3)!)^2+\nonumber\\
&~~~~~~~~~~~~~~~(n-a_0-1)(a_0(n-2)!+a_1(n-3)!)^2\big].
\end{align}
For both cases, if we denote $\eta_1=\frac{a_1}{n-1}$, then we have
\begin{align}\label{ineq03}
\mathcal{I}(\mathcal{T}_{0}(n,s))&\geq((n-1)!)^2\big\{(a_0+\eta_1^2)+\frac{(a_0^2 n+2a_0a_1-2a_0^2+a_1-a_0)}{(n-1)^2}+\nonumber\\
&~~~~~~~~~~~~~~~~~~~\frac{n-2-o(1)}{(n-1)^2}\big[a_0^2+a_0(a_0-1+\eta_1)^2+(n-a_0-1)(a_0+\eta_1)^2\big]\big\}.
\end{align}

To proceed the proof of Theorem \ref{theremovallemma}, we need some additional notations and a stability result by Ellis, Filmus and Friedgut \cite{EFF2015} (see Theorem 1 in \cite{EFF2015}). Assume each permutation in $S_n$ is distributed uniformly. Then, for a function $f:S_n\rightarrow\mathbb{R}$, the expected value of $f$ is defined by $\mathbb{E}[f]=\frac{1}{n!}\sum_{\sigma\in S_n}f(\sigma)$. The inner product of two functions $f,g:S_n\rightarrow\mathbb{R}$ is defined as $\langle f,g\rangle:=\mathbb{E}[f\cdot g]$, this induces the norm $\|f\|:=\sqrt{\langle f,f\rangle}$. Given $c>0$, denote $round(c)$ as the nearest integer to $c$.

\begin{theorem}\cite{EFF2015}\label{stabilityforper}
There exist positive constants $C_0$ and $\varepsilon_0$ such that the following holds. Let $\mathcal{F}$ be a subfamily of $S_n$ with $|\mathcal{F}|=\alpha(n-1)!$ for some $\alpha\leq\frac{n}{2}$. Let $f=\mathbbm{1}_{\mathcal{F}}$ be the characteristic function of $\mathcal{F}$ and let $f_{U_1}$ be the orthogonal projection of $f$ onto $U_1$. If $\mathbb{E}[(f-f_{U_1})^2]=\varepsilon\mathbb{E}[f]$ for some $\varepsilon\leq\varepsilon_0$, then
\begin{equation*}
\mathbb{E}[(f-g)^2]\leq C_0\alpha^2\left(\frac{1}{n^2}+\frac{\varepsilon^{\frac{1}{2}}}{n}\right),
\end{equation*}
where $g$ is the characteristic function of a union of $round(\alpha)$ cosets of $S_n$.
\end{theorem}

\begin{proof}[Proof of Theorem \ref{theremovallemma}]
For the convenience of our proof, for $\sigma,\pi\in S_n$, we denote $\sigma\cap \pi=\{i\in [n]:\sigma(i)=\pi(i)\}$. Set $c=\min\{\frac{\varepsilon_0}{12},\frac{1}{2}\}$ and $C=3C_0$, where $\varepsilon_0$ and $C_0$ are the positive constants from Theorem \ref{stabilityforper}. Let $f$ be the characteristic vector of $\mathcal{F}$. Write $f=f_0+f_1+f_2$, where $f_0$ is the projection of $f$ onto the eigenspace $U_{(n)}$ and $f_1$ is the projection of $f$ onto the eigenspace $U_{(n-1,1)}$. By the orthogonality of the eigenspaces, we have
\begin{equation}\label{ineq04}
\|f\|^2=\|f_0\|^2+\|f_1\|^2+\|f_2\|^2.
\end{equation}
Moreover, since $f$ is Boolean and $U_{(n)}=span\{\vec{\mathbf{1}}\}$, we also have
\begin{equation}\label{ineq05}
\begin{cases}
\|f\|^2=\mathbb{E}[f^2]=\mathbb{E}[f]=\frac{|\mathcal{F}|}{n!}=\frac{k+\varepsilon}{n},\\
\|f_0\|^2=\langle f,\vec{\mathbf{1}}\rangle^{2}=\mathbb{E}[f]^2=\frac{(k+\varepsilon)^2}{n^2}.
\end{cases}
\end{equation}

By the definition of $\mathcal{I}(\mathcal{F})$, we have
\begin{align}\label{ineq06}
\mathcal{I}(\mathcal{F})&=\sum_{\sigma\in \mathcal{F}}\sum_{\pi\in \mathcal{F}}|\sigma\cap\pi|=f^tBf,
\end{align}
where $B=(b_{i,j})_{n!\times n!}$ is a matrix with entry $b_{i,j}=|\sigma_i\cap\sigma_j|$ under a certain ordering of all permutations in $S_n=\{\sigma_1,\ldots,\sigma_{n!}\}$. According to the definition of $B$, we can write $B=\sum_{s=1}^{n}B_s$, where $B_s=(b^{s}_{i,j})_{n!\times n!}$ is the matrix with entries
\begin{equation*}
b^{s}_{i,j}=
\begin{cases}
1,~\text{if}~|\sigma_i\cap\sigma_j|\geq s;\\
0,~\text{otherwise}.
\end{cases}
\end{equation*}
From a simple observation, we know that $B_s=J-A_s$, where $J$ is the $n!\times n!$ matrix with all entries $1$ and $A_s$ is the adjacency matrix of $\Gamma_s$, i.e., the adjacency matrix of the Cayley graph on $S_n$ with generating set $FPF_s$. Therefore, by (\ref{ineq06}), we have
\begin{align}\label{ineq07}
\mathcal{I}(\mathcal{F})&=f^t\sum_{s=1}^{n}(J-A_s)f=nf^tJf-\sum_{s=1}^{n}f^tA_sf \nonumber\\
&=n|\mathcal{F}|^2-\sum_{s=1}^{n}(f_0^tA_sf_0+f_1^tA_sf_1+f_2^tA_sf_2).
\end{align}
Since $U_{(n)}$ and $U_{(n-1,1)}$ are eigenspaces for all $A_s$, $1\leq s\leq n$, therefore,
\begin{align*}
\mathcal{I}(\mathcal{F})&=n|\mathcal{F}|^2-n!\sum_{s=1}^{n}(\lambda_{(n)}^{(s)} \|f_0\|^2+\lambda_{(n-1,1)}^{(s)}\|f_1\|^2)-\sum_{s=1}^{n}f_2^tA_sf_2 \nonumber\\
&=n|\mathcal{F}|^2-n!\sum_{s=1}^{n}\big[(\lambda_{(n)}^{(s)}-\lambda_{(n-1,1)}^{(s)})\|f_0\|^2+\lambda_{(n-1,1)}^{(s)}\|f\|^2-\lambda_{(n-1,1)}^{(s)}\|f_2\|^2\big]-\sum_{s=1}^{n}f_2^tA_sf_2.
\end{align*}
According to Lemma \ref{cor1}, $\sum_{s=1}^{n}(\lambda_{(n)}^{(s)}-\lambda_{(n-1,1)}^{(s)})=n!(n-\frac{n-2}{n-1})$ and $\sum_{s=1}^{n}\lambda_{(n-1,1)}^{(s)}=(n-1)!(\frac{n-2}{n-1}-2)$. Therefore, we have
\begin{align}\label{ineq08}
\mathcal{I}(\mathcal{F})&=n|\mathcal{F}|^2-((n-1)!)^2\big[(k+\varepsilon)^2(n-\frac{n-2}{n-1})+(k+\varepsilon)(\frac{n-2}{n-1}-2)\big]-\nonumber\\
&~~~~((n-1)!)^2(2n-\frac{n^2-2n}{n-1})\|f_2\|^2-\sum_{s=1}^{n}f_2^tA_sf_2\nonumber\\
&=((n-1)!)^2\big[(k+\varepsilon)^2\frac{n-2}{n-1}+(k+\varepsilon)\frac{n}{n-1}\big]-\frac{((n)!)^2}{n-1}\|f_2\|^2-\sum_{s=1}^{n}f_2^tA_sf_2.
\end{align}

On the other hand, write $k+\varepsilon=a_0+\eta_1+\frac{c}{n-1}$ for some $0\leq c\leq 1$. By (\ref{ineq03}) and (\ref{ineq08}), we have
\begin{align*}
\mathcal{I}(\mathcal{F})-\mathcal{I}(\mathcal{T}_{0}(n,s))&\leq(\eta_1-\eta_1^2+\frac{c'}{n-1})((n-1)!)^2-\frac{((n)!)^2}{n-1}\|f_2\|^2-\sum_{s=1}^{n}f_2^tA_sf_2,
\end{align*}
where $c'=(1+2c)(a_0+\eta_1+1)$. Note that $\mathcal{I}(\mathcal{F})\geq \mathcal{I}(\mathcal{T}_{0}(n,s))-\delta((n-1)!)^2$, which indicates that
\begin{align}\label{ineq09}
\frac{(n!)^2}{n-1}\|f_2\|^2+\sum_{s=1}^{n}f_2^tA_sf_2\leq(\eta_1-\eta_1^2+\delta+\frac{c'}{n-1})((n-1)!)^2.
\end{align}
To obtain an upper bound on $\|f_2\|^2$, we need the following claim.

\textbf{Claim 1.} $\sum_{s=1}^{n}f_2^tA_sf_2\geq-(c_3\frac{(n!)^{2}}{n^2}+\frac{6(n!)^{2}}{n\ln{n}})\|f_2\|^{2}$ for some absolute constant $c_3$.
\begin{proof}
Denote $\Phi=\{\rho\vdash n:\rho\neq (n),(n-1,1)\}$. First, note that $f_2$ lies in $U_1^{\bot}$ and the eigenvalues corresponding to $U_1^{\bot}$ are $\{\lambda_{\rho}^{(s)}: \rho\in\Phi, 1\leq s\leq n\}$. Thus, we have
\begin{equation}\label{ineq10}
f_2^tA_sf_2=n!\sum_{\rho\in \Phi}\lambda_{\rho}^{(s)}\|f_{\rho}\|^2,
\end{equation}
where $f_{\rho}$ is the orthogonal projection of $f_2$ (or $f$) onto $U_{\rho}$.

Based on estimations about $\lambda_{\rho}^{(s)}$s for $\rho\in\Phi$ from Lemma \ref{lem1}, we have
\begin{align}\label{ineq11}
\begin{cases}
f_2^tA_sf_2\geq -c_3\frac{(n!)^{2}}{n^3}\|f_2\|^{2}, ~\text{for }3\leq s\leq n-\frac{n}{\ln{n}}-7;\\
f_2^tA_sf_2\geq -\frac{3(n!)^{2}}{n^2}\|f_2\|^{2}, ~\text{for } s=1,2 ~\text{and } n-\frac{n}{\ln{n}}-7\leq s\leq n,
\end{cases}
\end{align}
where $c_3>0$ is an absolute constant. This leads to
\begin{equation*}
\sum_{s=1}^{n}f_2^tA_sf_2\geq -(c_3\frac{(n!)^{2}}{n^2}+\frac{6(n!)^{2}}{n\ln{n}})\|f_2\|^{2}.
\end{equation*}
\end{proof}

Now, with the help of Claim 1 and (\ref{ineq09}), we have
\begin{align*}
(\eta_1-\eta_1^2+\delta+\frac{c'}{n-1})((n-1)!)^2&\geq\frac{(n!)^2}{n-1}\|f_2\|^2+\sum_{s=1}^{n}f_2^tA_sf_2\\
&\geq (n!)^2(\frac{1}{n-1}-\frac{7}{n\ln{n}})\|f_2\|^2\\
&\geq \frac{n!(n-1)!}{1+o(1)}\|f_2\|^2.
\end{align*}
From the definition, $\min\{\eta_1,1-\eta_1\}\leq |\varepsilon|$ and $c'\leq 3(k+\varepsilon+1)$. Thus, we have
\begin{align*}
\|f_2\|^2\leq\frac{\eta_1-\eta_1^2+\delta+\frac{c'}{n-1}}{n}(1+o(1))\leq \frac{|\varepsilon|+\delta}{k+\varepsilon}(1+o(1))\|f\|^2.
\end{align*}
Since $\max\{|\varepsilon|,\delta\}\leq ck$, we have
\begin{equation*}
\mathbb{E}[(f-f_{U_1})^2]=\|f_2\|^2\leq \varepsilon_0\|f\|^2=\varepsilon_0\mathbb{E}[f].
\end{equation*}
By Theorem \ref{stabilityforper}, there exists $\mathcal{G}$, a union of $k$ $1$-cosets of $S_n$ such that
\begin{equation*}
\mathbb{E}[(f-\mathbbm{1}_{\mathcal{G}})^2]\leq C_0(k+\varepsilon)^2\left(\sqrt{\frac{|\varepsilon|+\delta}{(k+\varepsilon)n^2}}(1+o(1))+\frac{1}{n^2}\right).
\end{equation*}
This leads to $|\mathcal{F}\Delta\mathcal{G}|=\mathbb{E}[(f-\mathbbm{1}_{\mathcal{G}})^2]\cdot n!\leq C_0(\sqrt{2k(|\varepsilon|+\delta)}+\frac{k}{n})|\mathcal{F}|$.

When $\varepsilon=\delta=0$, we have $k+\varepsilon=k=a_0$ and $\eta_1=0$. Since $0=a_1<a_0$ for this case, we need another estimation of $\mathcal{I}(\mathcal{T}_{0}(n,a_0(n-1)!))$. Similar to (\ref{ineq02}), we have
\begin{align}\label{ineq11.5}
\mathcal{I}(\mathcal{T}_{0}(n,s))&=\sum_{i,j\in [n]}|\mathcal{T}_{0}(n,s)_{i\rightarrow j}|^{2}=\sum_{j\in [n]}|\mathcal{T}_{0}(n,s)_{1\rightarrow j}|^{2}+\sum_{i\in [2,n]}\sum_{j\in [n]}|\mathcal{T}_{0}(n,s)_{i\rightarrow j}|^{2}\nonumber\\
&=a_0((n-1)!)^2+(n-1)\big[a_0((a_0-1)(n-2)!)^2+(n-a_0)(a_0(n-2)!)^2\big].
\end{align}
Therefore, combined with (\ref{ineq08}), (\ref{ineq11.5}) leads to
\begin{align}\label{ineq12}
\frac{(n!)^2}{n-1}\|f_2\|^2+\sum_{s=1}^{n}f_2^tA_sf_2\leq0.
\end{align}
By Claim 1, we have $\|f_2\|^2\leq 0$. Thus, $f=\mathbbm{1}_{\mathcal{F}}=f_0+f_1\in U_1$. As shown by Ellis et. al \cite{EFP2011} (see Theorem 7 and Theorem 8 in \cite{EFP2011}), this indicates that $\mathcal{F}$ is the union of $k$ $1$-cosets of $S_n$. Since $|\mathcal{F}|=k(n-1)!$, these $k$ $1$-cosets must be pairwise disjoint.

This completes the proof.
\end{proof}

\begin{Remark}
As an immediate consequence of Theorem \ref{theremovallemma}, when $|\varepsilon|,~\delta=o(\frac{1}{n})$, the optimal family $\mathcal{F}$ with maximal total intersection number is ``almost'' the union of $k$ disjoint $1$-cosets. However, due to the restrictions of parameters in Theorem \ref{stabilityforper}, the structural characterization given by Theorem \ref{theremovallemma} becomes weaker as each value of $|\varepsilon|$ and $\delta$ grows.
\end{Remark}

\section{Upper bounds on maximal total intersection numbers for families from different schemes}
In this section, we will show several upper bounds on maximal total intersection numbers for families of vector spaces and permutations using linear programming method for corresponding association schemes. 

\subsection{Grassmann scheme}\label{vector_space}
In this subsection, we take $(X,\mathcal{R})$ as the Grassmann scheme, which can be regarded as a $q$-analogue of the Johnson scheme (for explicit definition of Johnson scheme, see \cite{GC2016}).

For $1\leq k\leq n$, denote $G_q(n,k)$ as the set of all subspaces in $\mathbb{F}_q^n$ with constant dimension $k$ and $\mathcal{R}=\{R_0,\ldots,R_k\}$ as the corresponding distance relation set, where $R_i=\{(U_1,U_2)\in G_q(n,k)\times G_q(n,k):\dim(U_1\cap U_2)=k-i\}$.
$(G_q(n,k),\mathcal{R})$ is called $the$ $Grassmann$ $scheme$.

\begin{theorem}\label{thm5.1}\cite{DP1973}
Given $0\leq i,j\leq k$, the eigenvalues and the dual eigenvalues of the Grassmann scheme $G_q(n,k)$ are given by
\begin{flalign}
P_{i}(j)=E^{(q)}_{i}(j);\label{p_ploy} \\
Q_{j}(i)=D^{(q)}_{j}(i),\label{q_g}
\end{flalign}
where the $generalized~Eberlein~polynomial$ $E^{(q)}_i(x)$ and the $dual$ $Hahn$ $polynomial$ $D^{(q)}_j(x)$ are defined as follows:
\begin{flalign}
E_i^{(q)}(x)&=\sum\limits_{r=0}^i(-1)^{i-r}q^{i-r\choose 2}\Gaussbinom{k-r}{k-i}\Gaussbinom{k-x}{r}\Gaussbinom{n+r-k-x}{r}q^{rx},\label{gen_p}\\
D_j^{(q)}(x)&=\left(\Gaussbinom{n}{j}-\Gaussbinom{n}{j-1}\right)\sum\limits_{r=0}^j\left\{(-1)^{r}q^{r\choose 2}\Gaussbinom{j}{r}\Gaussbinom{n+1-r}{r}{\Gaussbinom{k}{r}}^{-1}{\Gaussbinom{n-k}{r}}^{-1}\right\}\Gaussbinom{x}{r}q^{-rx}.
\end{flalign}
\end{theorem}


Now, consider a family $\mathcal{F}\subseteq G_q(n,k)$ with size $M$.
According to the definition of $a_i$ in (\ref{inner_distribution}), we have
\begin{flalign*}
a_i=\frac{1}{M}|\{(U_1,U_2):~U_1,U_2\in \mathcal{F},~\dim(U_1\cap U_2)=k-i\}|.
\end{flalign*}
This leads to
\begin{flalign}\label{sum_of_Ai'_g}
a_{0}=1,~\sum\limits_{i=0}^k a_i=M.
\end{flalign}
Then, recall the definition of $\mathcal{I}(\mathcal{F})$ from (\ref{basic_id}), we have
\begin{flalign}
\mathcal{I}(\mathcal{F})&=M\sum\limits_{i=0}^k(k-i)a_{i}=kM\sum\limits_{i=0}^ka_i-M\sum\limits_{i=0}^kia_i\nonumber\\
&=kM^2-M\sum\limits_{i=0}^kia_i.\label{distance_and_A'_g}
\end{flalign}

%
Based on the relationship between inner distribution $a_i$s and dual distribution $b_i$s, we have the following theorem.
\begin{theorem}\label{d_B1}
Given a prime power $q$ and positive integers $n,k,M$ with $k\leq n$, $M\leq \Gaussbinom{n}{k}$. Let $\mathcal{F}\subseteq G_q(n,k)$ with $|\mathcal{F}|=M$ and $\{b_0,\ldots,b_k\}$ be the dual distribution of $\mathcal{F}$. Then, we have
\begin{flalign}
\mathcal{I}(\mathcal{F})&\le\left(b_1+1-\Gaussbinom{n}{1}\right)\frac{qM^2\Gaussbinom{k}{1}\Gaussbinom{n-k}{1}}{\Gaussbinom{n}{1}\left(\Gaussbinom{n}{1}-1\right)}+kM^2,\label{distance_and_B1'_grassmann}\\
\mathcal{I}(\mathcal{F})&\le\left(\frac{\Gaussbinom{n}{k}}{M}-\sum\limits_{r=2}^kb_r-\Gaussbinom{n}{1}\right)\frac{qM^2\Gaussbinom{k}{1}\Gaussbinom{n-k}{1}}{\Gaussbinom{n}{1}\left(\Gaussbinom{n}{1}-1\right)}+kM^2.\label{d_lp_bound_grassmann}
\end{flalign}
\end{theorem}
\begin{proof}
From (\ref{relationship}) and (\ref{A_represent_B}), we know that $b_1=\frac{1}{M}\sum\limits_{i=0}^kQ_1(i)a_i$. By (\ref{q_g}) and (\ref{gen_p}) from Theorem \ref{thm5.1}, we can further obtain
\begin{flalign*}
b_1&=\frac{1}{M}\sum\limits_{i=0}^k\left(\Gaussbinom{n}{1}-1\right)\left(1-\frac{\Gaussbinom{n}{1}\Gaussbinom{i}{1}}{{\Gaussbinom{k}{1}\Gaussbinom{n-k}{1}q^i}}\right)a_i\\
&\geq\frac{\left(\Gaussbinom{n}{1}-1\right)}{M}\sum\limits_{i=0}^ka_i-\frac{\Gaussbinom{n}{1}\left(\Gaussbinom{n}{1}-1\right)}{qM\Gaussbinom{k}{1}\Gaussbinom{n-k}{1}}\sum\limits_{i=0}^kia_i\\
&=\left(\Gaussbinom{n}{1}-1\right)-\frac{\Gaussbinom{n}{1}\left(\Gaussbinom{n}{1}-1\right)}{qM\Gaussbinom{k}{1}\Gaussbinom{n-k}{1}}\left(kM-\frac{\mathcal{I}(\mathcal{F})}{M}\right),
\end{flalign*}
where the last equality follows from (\ref{sum_of_Ai'_g}) and (\ref{distance_and_A'_g}). This leads to (\ref{distance_and_B1'_grassmann}).


On the other hand, from Lemma \ref{property_B'}, we know that $b_1=\frac{\Gaussbinom{n}{k}}{M}-1-\sum_{r=2}^kb_r$.
Thus, combined with $(\ref{distance_and_B1'_grassmann})$, this implies that
\begin{flalign*}
\mathcal{I}(\mathcal{F})&\le\left(b_1+1-\Gaussbinom{n}{1}\right)\frac{qM^2\Gaussbinom{k}{1}\Gaussbinom{n-k}{1}}{\Gaussbinom{n}{1}\left(\Gaussbinom{n}{1}-1\right)}+kM^2\\
&=\left(\frac{\Gaussbinom{n}{k}}{M}-\sum\limits_{r=2}^kb_r-\Gaussbinom{n}{1}\right)\frac{qM^2\Gaussbinom{k}{1}\Gaussbinom{n-k}{1}}{\Gaussbinom{n}{1}\left(\Gaussbinom{n}{1}-1\right)}+kM^2,
\end{flalign*}
which gives (\ref{d_lp_bound_grassmann}).
\end{proof}

With the help of Theorem \ref{d_B1}, we now proceed the proof of Theorem \ref{subspace_intersection}.

\begin{proof}[Proof of Theorem \ref{subspace_intersection}] From Lemma $\ref{B'_k}$, we know that $b_j\ge 0$ for $0\leq j\leq k$. This leads to $\sum_{r=2}^kb_r\ge 0$. Thus, combined with $(\ref{d_lp_bound_grassmann})$, we have
\begin{flalign*}
\mathcal{MI}(\mathcal{F})&\leq\left(\frac{\Gaussbinom{n}{k}}{M}-\sum\limits_{r=2}^kb_r-\Gaussbinom{n}{1}\right)\frac{qM^2\Gaussbinom{k}{1}\Gaussbinom{n-k}{1}}{\Gaussbinom{n}{1}\left(\Gaussbinom{n}{1}-1\right)}+kM^2\\
&\le \left(\frac{\Gaussbinom{n}{k}}{M}-\Gaussbinom{n}{1}\right)\frac{qM^2\Gaussbinom{k}{1}\Gaussbinom{n-k}{1}}{\Gaussbinom{n}{1}\left(\Gaussbinom{n}{1}-1\right)}+kM^2.
\end{flalign*}
This proves inequality $(\ref{lower_bound_0})$.

Next, we shall use a linear programming method to give a lower bound of $\sum_{r=2}^kb_s$. 
From Lemma $\ref{B_represent_A}$, we know that for $1\leq i\leq k$,
\begin{flalign}\label{A'ige0_subspace}
\sum\limits_{r=0}^kb_rP_i(r)\ge 0.
\end{flalign}
Meanwhile, by Lemma \ref{property_B'}, we also have
$b_0=1$ and $b_1=\frac{\Gaussbinom{n}{k}}{M}-1-\sum_{r=2}^k b_r$. Thus, this leads to
\begin{flalign}
\sum\limits_{r=0}^kb_rP_i(r)&=b_0P_i(0)+b_1P_i(1)+\sum\limits_{r=2}^kb_rP_i(r)\nonumber\\
&=P_i(0)+\left(\frac{\Gaussbinom{n}{k}}{M}-1-\sum\limits_{r=2}^kb_r\right)P_i(1)+\sum\limits_{r=2}^kb_rP_i(r)\nonumber\\
&=P_i(0)+\frac{\Gaussbinom{n}{k}}{M}P_i(1)-P_i(1)+\sum\limits_{r=2}^k[P_i(r)-P_i(1)]b_r.\label{lp_subspace}
\end{flalign}
Combining $(\ref{A'ige0_subspace})$ with $(\ref{lp_subspace})$, we further have
\begin{flalign*}
\sum\limits_{r=2}^kb_r[P_i(1)-P_i(r)]\leq P_i(0)+\frac{\Gaussbinom{n}{k}}{M}P_i(1)-P_i(1),
\end{flalign*}
for $1\leq i\leq k$. To obtain a lower bound on $\sum_{r=2}^{k}b_r$ under the restrictions of the above inequality together with $b_r\geq 0$ ($2\leq r\leq k$) from Lemma \ref{B'_k}, we now consider the following LP problem:

(I) Choose real variables $y_2,\ldots,y_k$ so as to
\begin{flalign*}
\Lambda(n,k,q,M)=\text{minimize }\sum\limits_{r=2}^ky_r,
\end{flalign*}
subject to
\begin{flalign*}
\begin{cases}
y_r\ge 0,~\text{for}~r=2,3,\ldots,k;\\
\sum\limits_{r=2}^ky_r[P_i(1)-P_i(r)]\leq P_i(0)+\frac{\Gaussbinom{n}{k}}{M}P_i(1)-P_i(1),~\text{for}~i=1,2,\ldots,k.
\end{cases}
\end{flalign*}

Note that when $M\ge \Gaussbinom{n-1}{k-1}$, by $(\ref{p_ploy})$ and $(\ref{gen_p})$, we have
\begin{flalign*}
1+\frac{\Gaussbinom{n}{k}}{M}\frac{P_i(1)}{P_i(0)}-\frac{P_i(1)}{P_i(0)}&=1+\left(\frac{\Gaussbinom{n}{k}}{M}-1\right)\frac{P_i(1)}{P_i(0)}\\
&=1+\left(\frac{\Gaussbinom{n}{k}}{M}-1\right)\left(1-\frac{\Gaussbinom{n}{1}\Gaussbinom{i}{1}}{\Gaussbinom{k}{1}\Gaussbinom{n-k}{1}q^i}\right)\\
&\ge1+\left(\frac{\Gaussbinom{n}{k}}{\Gaussbinom{n-1}{k-1}}-1\right)\left(1-\frac{\Gaussbinom{n}{1}}{\Gaussbinom{n-k}{1}q^k}\right)=0.
\end{flalign*}
Moveover, since $P_i(0)=\Gaussbinom{k}{1}\Gaussbinom{n-k}{1}q^{i^2}\ge 0$, this also leads to $P_i(0)+\frac{\Gaussbinom{n}{k}}{M}P_i(1)-P_i(1)\ge 0$, for $1\leq i\leq k$. Therefore, by taking $y_2=y_3=\cdots=y_k=0$, we can obtain the optimal solution $\Lambda(n,k,q,M)=0$.

When $M\leq \Gaussbinom{n-1}{k-1}$, by (\ref{d_lp_bound_grassmann}), for $\mathcal{F}\subseteq G_q(n,k)$ with $|\mathcal{F}|=M\leq \Gaussbinom{n-1}{k-1}$, we have:
\begin{flalign}
\mathcal{MI}(\mathcal{F})&\le\left(\frac{\Gaussbinom{n}{k}}{M}-\sum\limits_{r=2}^kb_r-\Gaussbinom{n}{1}\right)\frac{qM^2\Gaussbinom{k}{1}\Gaussbinom{n-k}{1}}{\Gaussbinom{n}{1}\left(\Gaussbinom{n}{1}-1\right)}+kM^2\nonumber\\
&\le\left(\frac{\Gaussbinom{n}{k}}{M}-\Lambda(n,k,q,M)-\Gaussbinom{n}{1}\right)\frac{qM^2\Gaussbinom{k}{1}\Gaussbinom{n-k}{1}}{\Gaussbinom{n}{1}\left(\Gaussbinom{n}{1}-1\right)}+kM^2.\label{d_lp_bound_formal_space}
\end{flalign}
Consider the dual problem of (I), which is given as follows (see \cite[Section 4 of Chapter 17]{MS1977}).

(II) Choose real variables $x_1,x_2,\dots,x_k$ so as to
\begin{flalign*}
{\bar{\Lambda}}(n,k,q,M)=\text{maximize }\sum\limits_{i=1}^k\left[P_i(1)-\frac{\Gaussbinom{n}{k}}{M}P_i(1)-P_i(0)\right]x_i,
\end{flalign*}
subject to
\begin{flalign*}
\begin{cases}
x_i\ge 0,~\text{for}~i=1,2,\ldots,k;\\
\sum\limits_{i=1}^k x_i[P_i(1)-P_i(r)]\ge -1,~\text{for}~ r=2,3,\ldots,k.
\end{cases}
\end{flalign*}
%
We claim that $x_1=\cdots=x_{k-1}=0$, $x_k=[P_{k}(2)-P_{k}(1)]^{-1}$ is a feasible solution to the above $LP$ problem (II). To show this, we only need to prove that
\begin{flalign}\label{c1}
\frac{P_k(1)-P_k(r)}{P_k(2)-P_k(1)}\ge -1,
\end{flalign}
for $2\leq r\leq k$. From $(\ref{gen_p})$ and $(\ref{p_ploy})$, we know that $P_k(r)=(-1)^rq^{{r\choose 2}+k(k-r)}\Gaussbinom{n-k-r}{k-r}$. Thereofore, $P_k(2)-P_k(1)> 0$ and $(\ref{c1})$ follows from the fact that $q^{{r\choose 2}+k(k-r)}\Gaussbinom{n-k-r}{k-r}$ is decreasing on $r$ when $k\leq n/2$. With this feasible solution, we have
\begin{flalign*}
{\bar{\Lambda}}(n,k,q,M)&\ge\frac{P_k(1)-\frac{\Gaussbinom{n}{k}}{M}P_k(1)-P_k(0)}{P_k(2)-P_k(1)}\\
&=\frac{-q^{k(k-1)}\Gaussbinom{n-k-1}{k-1}+\frac{\Gaussbinom{n}{k}}{M}q^{k(k-1)}\Gaussbinom{n-k-1}{k-1}-q^{k^2}\Gaussbinom{n-k}{k}}{q^{1+k(k-2)}\Gaussbinom{n-k-2}{k-2}+q^{k(k-1)}\Gaussbinom{n-k-1}{k-1}}\\
&=\left(\frac{\Gaussbinom{n}{k}}{M}-\frac{q^n-1}{q^k-1}\right)\frac{q^{k-1}(q^{n-k-1}-1)}{q^{n-2}-1}.
\end{flalign*}
Therefore, it follows from $(\ref{d_lp_bound_formal_space})$ that
\begin{flalign*}
\mathcal{MI}(\mathcal{F})&\le\left[\frac{\Gaussbinom{n}{k}}{M}-\left(\frac{\Gaussbinom{n}{k}}{M}-\frac{q^n-1}{q^k-1}\right)\frac{q^{k-1}(q^{n-k-1}-1)}{q^{n-2}-1}-\Gaussbinom{n}{1}\right]\frac{qM^2\Gaussbinom{k}{1}\Gaussbinom{n-k}{1}}{\Gaussbinom{n}{1}\left(\Gaussbinom{n}{1}-1\right)}+kM^2\\
&=\left[\frac{\Gaussbinom{n}{k}(q^{k-1}-1)}{M(q^{n-2}-1)}-\frac{(q^n-1)(q^{n-1}-1)(q^{k-1}-1)}{(q-1)(q^k-1)(q^{n-2}-1)}\right]\frac{M^2(q^k-1)(q^{n-k}-1)}{(q^n-1)(q^{n-1}-1)}+kM^2\\
&=\left[\frac{\Gaussbinom{n}{k}}{M}-\frac{(q^n-1)(q^{n-1}-1)}{(q-1)(q^k-1)}\right]\frac{M^2(q^k-1)(q^{k-1}-1)(q^{n-k}-1)}{(q^n-1)(q^{n-1}-1)(q^{n-2}-1)}+kM^2.
\end{flalign*}
This completes the proof of $(\ref{lower_bound_space_2})$.
\end{proof}

As an immediate consequence of Theorem \ref{subspace_intersection} and Proposition \ref{counting intersection subspace}, we have the following corollaries showing that bounds in Theorem \ref{subspace_intersection} are tight for some special cases.
\begin{corollary}
Given a prime power $q$, a positive integer $n$ with $n\ge 2$, for $\mathcal{F}\subseteq G_{q}(n,2)$ with $|\mathcal{F}|=\Gaussbinom{n-1}{1}$, we have
\begin{flalign}\label{subspace_1_star_2}
\mathcal{MI}\left(\mathcal{F}\right)=\frac{(q^{n-1}+q-2)(q^{n-1}-1)}{(q-1)^2}.
\end{flalign}
\end{corollary}
\begin{proof}
By inequality ($\ref{lower_bound_0}$), we already have
\begin{flalign}
\mathcal{MI}\left(\mathcal{F}\right)\leq\frac{(q^{n-1}+q-2)(q^{n-1}-1)}{(q-1)^2}.\nonumber
\end{flalign}
To show that this upper bound is tight, we take $\mathcal{Y}_1\subseteq G_q(n,2)$ as the family of all $2$-dim subspaces containing some fixed $1$-dim subspace. Clearly, $|\mathcal{Y}_1|=\Gaussbinom{n-1}{1}$. By Proposition \ref{counting intersection subspace}, we have
\begin{flalign}
\mathcal{I}(\mathcal{Y}_1)&=\Gaussbinom{n-1}{1}\sum\limits_{j=0}^1 (j+1)q^{(1-j)^2}\Gaussbinom{n-2}{1-j}\Gaussbinom{1}{j}\nonumber\\
&=\Gaussbinom{n-1}{1}\frac{q^{n-1}+q-2}{q-1}.\nonumber
\end{flalign}
Hence, $(\ref{subspace_1_star_2})$ follows from $\mathcal{MI}(\mathcal{F})\geq \mathcal{I}(\mathcal{Y}_1)$.
\end{proof}

\begin{corollary}
Given a prime power $q$, a positive integer $n$ with $n\ge 6$, for $\mathcal{F}\subseteq G_{q}(n,3)$ with $|\mathcal{F}|=\Gaussbinom{n-2}{1}$, we have
\begin{flalign}\label{subspace_2_star_3}
\mathcal{MI}\left(\mathcal{F}\right)=\frac{(2q^{n-2}+q-3)(q^{n-2}-1)}{(q-1)^2}.
\end{flalign}
\end{corollary}
\begin{proof}
Similarly, by $(\ref{lower_bound_space_2})$, we have
\begin{flalign}
\mathcal{MI}\left(\mathcal{F}\right)\leq\frac{(2q^{n-2}+q-3)(q^{n-2}-1)}{(q-1)^2}.\nonumber
\end{flalign}
Now, take $\mathcal{Y}_2\subseteq G_q(n,3)$ as the family of all $3$-dim subspaces containing some fixed $2$-dim subspace. Clearly, $|\mathcal{Y}_2|=\Gaussbinom{n-2}{1}$. By Proposition \ref{counting intersection subspace}, we have
\begin{flalign}
\mathcal{I}(\mathcal{Y}_2)&=\Gaussbinom{n-2}{1}\sum\limits_{j=0}^1 (j+2)q^{(1-j)^2}\Gaussbinom{n-3}{1-j}\Gaussbinom{1}{j}\nonumber\\
&=\frac{(2q^{n-2}+q-3)(q^{n-2}-1)}{(q-1)^2}.\nonumber
\end{flalign}
Hence, $(\ref{subspace_2_star_3})$ follows from $\mathcal{MI}(\mathcal{F})\geq \mathcal{I}(\mathcal{Y}_2)$.
\end{proof}

\subsection{The conjugacy scheme of symmetric group}\label{permutation}


Given a positive integer $n$, we take $X$ as the symmetric group $S_n$. Denote $C_0,C_1,\ldots,C_s$ as the conjugacy classes of $S_n$ and the relations $\mathcal{R}=\{R_0,\ldots,R_s\}$ are defined as follows:
\begin{flalign*}
R_i=\{(g,h)\in S_n\times S_n|~gh^{-1}\in C_i\}.
\end{flalign*}
$(S_n,\mathcal{R})$ is called $the~conjugacy~scheme$ of $S_n$.
For each element $\sigma$ of $S_n$, one can write
\begin{flalign*}
\sigma=(a_1\ldots a_{k_1})(b_1\ldots b_{k_2})\ldots(c_1\ldots c_{k_m}),
\end{flalign*}
as a product of disjoint cycles with $k_1\ge k_2\ge\ldots \geq k_m\ge 1$. This $m$-tuple $(k_1,\ldots,k_m)$ is called the cycle-shape of $\sigma$. 
Then, the conjugacy classes of $S_n$ are precisely
\begin{flalign*}
\{\sigma\in S_n:~\text{cycle-shape}(\sigma)=\lambda\}_{\lambda\vdash n}.
\end{flalign*}
Clearly, each conjugacy class $\{C_i:~0\leq i\leq s\}$ corresponds to a cycle-shape $\sigma_i$ of $S_n$ respectively. In particular, $C_0$ corresponds to the cycle-shape $(1,1,\ldots,1)$. According to \cite[Chapter 11.12]{GC2016}, eigenvalues and dual eigenvalues of the conjugacy scheme of $S_n$ are given by
\begin{flalign}\label{conjugucypq}
P_{i}(j)=\frac{|C_i|\overline{\psi_j(c_i)}}{\psi_j(e_0)},~~Q_{j}(i)=\psi_j(c_i)\psi_j(e_0),
\end{flalign}
where $c_i\in C_i$ for $0\leq i\leq s$, $e_0$ is the identity element in $S_n$ and $\psi_j$ $(0\leq j\leq s)$ denote irreducible characters of $S_n$. Especially, $\psi_0$ denotes the trivial character, which maps all the elements of $G$ into $1$.

Given $\mathcal{F}\subseteq S_n$ with size $M$, consider the inner distribution of $\mathcal{F}$ with respect to $\mathcal{R}$. According to the definition of $a_i$, $(\ref{inner_distribution})$ can be rewritten as
\begin{flalign*}
a_i=\frac{1}{M}|\{(x,y):~x,y\in \mathcal{F},~xy^{-1}\in C_i\}|.
\end{flalign*}
%
Thus, one can easily obtain
\begin{flalign}\label{sum_of_Ai'_con}
a_{0}=1~\text{and}~\sum\limits_{i=0}^s a_i=M.
\end{flalign}
Given a cycle-shape $\sigma=(k_1,\ldots,k_m)$, define $U_\sigma=|\{i\in[m]:k_i=1\}|$. From the new expression of $a_i$ above, we have
\begin{flalign}\label{distance_and_A'_con}
\mathcal{I}(\mathcal{F})=M\sum\limits_{i=1}^s U_{\sigma}a_{i}.
\end{flalign}
%

Now, according to the relationship between inner distribution $a_i$s and dual distribution $b_i$s, we have the following theorem.
\begin{theorem}\label{LP_permutation}
Given positive integers $n$ and $M$ with $M\leq n!$. Let $\mathcal{F}\subseteq S_n$ with $|\mathcal{F}|=M$ and $\{b_0,\ldots,b_s\}$ be the dual distribution of $\mathcal{F}$. Then, we have
\begin{flalign}
\mathcal{I}(\mathcal{F})&=M^2\left(\frac{b_1}{n-1}+1\right),\label{distance_and_B1'}\\
\mathcal{I}(\mathcal{F})&=\frac{M^2}{n-1}\left(\frac{n!}{M}+n-2-\sum\limits_{r=2}^sb_r\right).\label{d_lp_bound}
\end{flalign}
\end{theorem}
\begin{proof}
According to \cite[Lemma 6.9]{CG1978}, there exists an irreducible character $\psi$ of $S_n$ which is defined as: $\psi(c)=U_{\sigma(c)}-1$ for $c\in S_n$, where $\sigma(c)$ is the cycle-shape of $c$. W.l.o.g, we can assume that $\psi_1=\psi$. By (\ref{relationship}) and (\ref{A_represent_B}), we know that $b_1=\frac{1}{M}\sum\limits_{i=0}^sQ_1(i)a_i$. Then, by (\ref{conjugucypq}), we further have
\begin{flalign*}
b_1&=\frac{1}{M}\sum\limits_{i=0}^s\psi_1(c_i)\psi_1(e_0)a_i\\
&=\frac{1}{M}\sum\limits_{i=0}^s(U_{\sigma(c_i)}-1)(n-1)a_i.
\end{flalign*}
Combined with (\ref{sum_of_Ai'_con}) and (\ref{distance_and_A'_con}), this leads to
\begin{flalign*}
b_1&=\frac{1}{M}\left(\sum\limits_{i=0}^s(n-1)a_i U_{\sigma(c_i)}-\sum\limits_{i=0}^s(n-1)a_i\right)\\
&=\frac{1}{M^2}(n-1)\mathcal{I}(\mathcal{F})-(n-1),
\end{flalign*}
Therefore, we have (\ref{distance_and_B1'}).

On the other hand, by Lemma \ref{property_B'}, we have $b_1=\frac{n!}{M}-1-\sum_{r=2}^sb_r$.
Thus, combined with (\ref{distance_and_B1'}), this implies that
\begin{flalign*}
\mathcal{I}(\mathcal{F})&=M^2\left(\frac{b_1}{n-1}+1\right)=\frac{M^2}{n-1}\left(\frac{n!}{M}+n-2-\sum\limits_{r=2}^sb_r\right).
\end{flalign*}
\end{proof}
\begin{proof}[Proof of Theorem \ref{lpbound}]\par
From Lemma $\ref{B'_k}$, $b_j\ge 0$ for $0\leq j\leq s$. This leads to $\sum_{r=2}^sb_r\ge 0$. Thus, combined with $(\ref{d_lp_bound})$, we have
\begin{flalign*}
\mathcal{MI}(\mathcal{F})&=\frac{M^2}{n-1}\left(\frac{n!}{M}+n-2-\sum\limits_{r=2}^sb_r\right)\\
&\le \frac{M^2}{n-1}\left(\frac{n!}{M}+n-2\right).
\end{flalign*}
\end{proof}
\begin{Remark}
Actually, similar to the proof of Theorem \ref{subspace_intersection}, we can also use the linear programming approach to bound $\sum_{r=2}^sb_r$. For interested readers, the corresponding LP problem is formulated as follows:

(I) Choose real variables $y_2,\ldots,y_k$ so as to
\begin{flalign*}
\Lambda(n,M)=\text{minimize }\sum\limits_{r=2}^sy_r,
\end{flalign*}
subject to
\begin{flalign*}
\begin{cases}
y_r\ge 0,~\text{for}~r=2,3,\ldots,s;\\
\sum\limits_{r=2}^ky_r[P_i(1)-P_i(r)]\leq P_i(0)+\frac{n!}{M}P_i(1)-P_i(1),~\text{for}~i=1,2,\ldots,s.
\end{cases}
\end{flalign*}
Note that when $M\ge (n-1)!$, the optimal solution $\Lambda(n,M)=0$ is given by taking $y_2=y_3=\cdots=y_s=0$. When $M\le (n-1)!$, we turn to the following the dual problem of (I).

(II) Choose real variables $x_1,x_2,\dots,x_s$ so as to
\begin{flalign*}
\overline{\Lambda}(n,M)=\text{maximize }\sum\limits_{i=1}^s\left[P_i(1)-\frac{n!}{M}-P_i(0)\right]x_i,
\end{flalign*}
subject to
\begin{flalign*}
\begin{cases}
x_i\ge 0,~\text{for}~i=1,2,\ldots,s;\\
\sum\limits_{i=1}^s x_i[P_i(1)-P_i(r)]\ge -1,~\text{for}~r=2,3,\ldots,s.
\end{cases}
\end{flalign*}

Unfortunately, the feasible solution we find is $x_1=\ldots=x_s=0$, which leads to the same lower bound $\sum_{r=2}^sb_r\geq 0$ as Lemma \ref{B'_k}. Possibly, one can find other more proper feasible solutions to improve Theorem \ref{lpbound}.
\end{Remark}

As an immediate consequence of Theorem \ref{lpbound}, we have the following corollary.
\begin{corollary}
Given a positive integer $n\ge 2$, let $\mathcal{F}\subseteq S_n$ with $|\mathcal{F}|=(n-1)!$, we have
\begin{flalign}\label{d_n-1}
\mathcal{MI}\left(\mathcal{F}\right)=2\left((n-1)!\right)^2.
\end{flalign}
\end{corollary}
\begin{proof}
By Theorem $\ref{lpbound}$, we have
\begin{flalign*}
\mathcal{MI}\left(\mathcal{F}\right)\leq 2\left((n-1)!\right)^2.
\end{flalign*}
On the other hand, by taking $\mathcal{Y}=\{y\in S_n: y(1)=1\}\subseteq S_n$, we have $\mathcal{I}(\mathcal{Y})=2\left((n-1)!\right)^2$. Therefore, (\ref{d_n-1}) follows from $\mathcal{MI}(\mathcal{F})\geq \mathcal{I}(\mathcal{Y})$.
\end{proof}

\section{Concluding remarks}

In this paper, we consider a new type of intersection problems which can be viewed as inverse problems of Erd\H{o}s-Ko-Rado type theorems for vector spaces and permutations. This type of problems concerns extremal structures of the families of subspaces or permutations with maximal total intersection numbers. Through different methods, we obtain structural characterizations of the optimal family of subspaces and the optimal family of permutations satisfying $\mathcal{I}(\mathcal{F})=\mathcal{MI}(\mathcal{F})$. To some extent, these results determine the unique structure of the optimal families for certain values of $|\mathcal{F}|$ and characterize the relation between maximizing $\mathcal{I}(\mathcal{F})$ and being intersecting, which partially answers Question \ref{question1}. Moreover, using linear programming methods, we have also shown several upper bounds on $\mathcal{MI}(\mathcal{F})$. These bounds may provide a reference for the study of structures of optimal families.

However, our results have the following limits:
\begin{itemize}
  \item Take $\varepsilon_0=\frac{1}{96t\ln{q}(k+1)}$, Theorem \ref{stabilityforsub} shows that for $n$ large enough and $\delta\leq 1+\varepsilon_0$ not too close to $0$, the optimal family with maximal total intersection number is either contained in a full $t$-star or containing a full $t$-star. When $|\mathcal{F}|>(1+\varepsilon_0)\Gaussbinom{n-t}{k-t}$, the quantitative shifting arguments in the proof of Theorem \ref{stabilityforsub} no longer work. So, can we obtain similar structural results for families with size larger than $(1+\varepsilon_0)\Gaussbinom{n-t}{k-t}$? Note that the intersection problem of vector spaces often requires tools from linear algebra or exterior algebra, maybe ideas from these areas can help us to tackle this problem.
  \item For families of permutations, we consider the case for $|\mathcal{F}|=\Theta((n-1)!)$. Nevertheless, for $|\mathcal{F}|=\Theta((n-t)!)$ ($t\geq 2$), nothing is known yet. It is worth noting that, in \cite{EFF2017}, the authors provide a stability result for families of permutations with size $\Theta((n-t)!)$ similar to Theorem \ref{stabilityforper}. Thus, it is natural to wonder if we can extend the result of Theorem \ref{theremovallemma} to families with size $|\mathcal{F}|=\Theta((n-t)!)$ using this stability result. Sadly, this requires more information about spectra of $\Gamma_k$, which is beyond our capability.

      ~~~~Moreover, when $\varepsilon$ becomes relatively large, the result of Theorem \ref{theremovallemma} seems to be trivial. Thus, for this case, more specific structural characterizations for families of permutations are also worth studying.
  \item Due to the choice of feasible solutions of corresponding LP problems, our upper bounds on $\mathcal{MI}(\mathcal{F})$ are no longer tight for families of subspaces with size $\Theta(\Gaussbinom{n-t}{k-t})$ or families of permutations with size $\Theta((n-t)!)$, for $t\geq 2$. Therefore, one can try to find other more proper feasible solutions to improve these upper bounds.
\end{itemize}

\subsection*{Acknowledgements}

The authors express their gratitude to the two anonymous reviewers for their detailed and constructive comments which are
very helpful to the improvement of the presentation of this paper, especially for providing a simpler proof of Proposition \ref{ob1} using the method of generating functions. And the authors also express their gratitude to the associate editor for his excellent editorial job.

\bibliographystyle{abbrv}
\bibliography{REF}
\end{document}